\documentclass{amsart}

\usepackage{xcolor}
\usepackage[verbose,hypertexnames=false]{hyperref}
\hypersetup{colorlinks=false,allbordercolors=blue,pdfborderstyle={/S/U/W 1}}

\usepackage{amsmath,amssymb}
\usepackage{mathtools}

\usepackage{booktabs}
\usepackage{tikz}
\usepackage{float}
\usepackage{xcolor}
\newtheorem{theorem}{Theorem}[section]
\newtheorem{lemma}[theorem]{Lemma}

\newtheorem{corollary}[theorem]{Corollary}

\theoremstyle{definition}
\newtheorem{definition}[theorem]{Definition}
\newtheorem{example}[theorem]{Example}
\newtheorem{remark}[theorem]{Remark}

\begin{document}

\title{S-Equivalence of Band-Twisted Genus One Knot}

\author{Jun Wang}

\address{School of Mathematical Sciences, Hebei Normal University \\
wjun@hebtu.edu.cn}

\author{Ziyi Liu}

\address{School of Mathematical Sciences, Hebei Normal University\\
ziyiliu.topo@outlook.com}

\maketitle

\begin{abstract}
We add twists to a band of a genus-one Seifert surface, producing
	a knot $K(\ell,0)$. We show that $K$ and $K(\ell,0)$ are $\Lambda_1$-equivalent (a unimodular basis change), hence $S$-equivalent precisely when the $(2,2)$-entry of the Seifert matrix of $K$ vanishes and the sum $s=a_{12}+a_{21}$ of its off-diagonal entries divides $\ell$; the weaker $S$-equivalence holds under a broader congruence governed by the subgroup $\mathcal{S}^+_{s^2}$ of Aka--Feller--Miller--Wieser.
	The $\Lambda_1$-equivalence criterion is established by an $SL_2(\mathbb{Z})$-analysis and an explicit $\Lambda_1$-operation; the precise $S$-equivalence condition follows from the framework of Aka--Feller--Miller--Wieser. The Jones polynomial distinguishes the
	knots when $V(K)\neq 1$, yielding infinite families of $S$-equivalent
	but inequivalent genus-one knots, illustrated by $9_{46}$.
The pairs we construct are explicit examples toward the surjectivity direction of Problem~1.6 in Kirby's list (K3) and relate to Problem~7.7 of Aka--Feller--Miller--Wieser: they exhibit $S$-equivalent genus-one Seifert matrices of the same dimension carried by inequivalent knots.\end{abstract}


	\section{Introduction}\label{sec:intro}
	
	Two integral matrices are said to be \emph{$S$-equivalent} \cite{Murasugi1965} if they can be transformed into one another by a sequence of elementary operations $\Lambda_i^{\pm 1}$ ($i=1,2,3$), where $\Lambda_1$ is conjugation by a unimodular matrix (a basis change), and $\Lambda_2$, $\Lambda_3$ enlarge the matrix by adding a row and a column. K. Murasugi \cite{Murasugi1965} proved that the Seifert matrices of the same knot are $S$-equivalent.  In 2003,  Naik and Stanford \cite{NaikStanford2003} proved that $S$-equivalence is generated by the doubled-delta move on knot diagrams, which in general changes the knot type. 
	
	C.\ Livingston proposed the following problem, which is Problem~1.6 in Kirby's list (K3) \cite{BaykurKirbyRuberman2026}:
	
	\begin{quote}
		\textbf{Problem 1.6.} Suppose that $M_1$ and $M_2$ are $S$-equivalent Seifert matrices. Does there exist a fixed knot bounding Seifert surfaces $F_1$ and $F_2$ for which the associated matrices are $M_1$ and $M_2$?
	\end{quote}
	
	As the remarks to Problem~1.6 observe, the motivating question is the surjectivity direction: given an $S$-equivalence class, can one find a knot whose matrices exhaust that class? While the Alexander-polynomial-one $S$-equivalence class provides trivial non-surjectivity examples (any non-trivial knot with $\Delta_K(t)=1$ cannot realize the empty matrix), constructing explicit pairs of $S$-equivalent but inequivalent knots with Seifert matrices of the \emph{same} dimension remains a natural intermediate step toward understanding Problem~1.6.
	
We describe a simple geometric operation on a genus-one Seifert surface---adding twists to a band---and determine precisely when the resulting knot $K(\ell,0)$ is $\Lambda_1$-equivalent to $K$, and more generally when it is $S$-equivalent. Here $2\ell$ twists on the band change the $(1,1)$-entry of the Seifert matrix from $a_{11}$ to $a_{11}-\ell$ (see Section~\ref{construction}).
\begin{theorem}\label{thm:main}
	Let $K$ be a genus-one knot with Seifert matrix $M = \begin{pmatrix} a_{11} & a_{12} \\ a_{21} & a_{22} \end{pmatrix}$ in the canonical band basis, and let $K(\ell,0)$ be the knot obtained by adding $2\ell$ twists to the first band ($\ell\neq0$). 
	Then $K$ and $K(\ell,0)$ are $\Lambda_1$-equivalent if and only if $a_{22}=0$ and $(a_{12}+a_{21})\mid\ell$.
\end{theorem}
	The necessity of $a_{22}=0$ follows from the Alexander polynomial; the necessity of $s\mid\ell$ for $\Lambda_1$-equivalence is established by Lemma~\ref{lem:sequiv}, and sufficiency is the explicit $\Lambda_1$-operation of Lemma~\ref{lem:sequiv}. By Lemma~\ref{lem:jones}, when $V(K)\neq 1$ the knots $K$ and $K(\ell,0)$ are $S$-equivalent but inequivalent; in particular they are distinct knots. This yields infinite families parametrized by $\ell$. For example, the knot $9_{46}$ (Rolfsen's table) has Seifert matrix $\begin{pmatrix}0&2\\1&0\end{pmatrix}$, so $s=3$. Since $\mathcal{S}^+_{9}$ is trivial, $K$ and $K(\ell,0)$ are $\Lambda_1$- and $S$-equivalent precisely when $3\mid\ell$, giving four distinct companions $K(\pm 3,0)$ and $K(0,\pm 3)$ all distinguished from $9_{46}$ by the Jones polynomial (Section~\ref{example}).

	 The pair of Seifert matrices of $K$ and $K(\ell,0)$ is therefore an explicit example of $S$-equivalent genus-one matrices of the same dimension carried by different knots.  The weaker notion of $S$-equivalence is characterized in Theorem~\ref{thm:sequiv-char}: it holds under a broader congruence governed by the subgroup $\mathcal{S}^+_{s^2}$ of Aka--Feller--Miller--Wieser \cite{AkaFellerMillerWieser2023}, so $\Lambda_1$-equivalence implies $S$-equivalence ($s=1,3$) but the converse fails once $\mathcal{S}^+_{s^2}$ is non-trivial (e.g.\ $s=5$). 
	
	As a corollary of Theorem \ref{thm:sequiv-char}, we obtain that 
	\begin{corollary}
		Let $a_{12}+a_{21}\in\{1,3\}$. For two Seifert matrices $M_1=\begin{pmatrix}a_{11}&a_{12}\\a_{21}&a_{22}\end{pmatrix}$ and $M_2=\begin{pmatrix}a_{11}-\ell&a_{12}\\a_{21}&a_{22}\end{pmatrix},\ell\neq 0$, if they are $S$-equivalent, then there shall exist a knot bounding Seifert surfaces $F_1$ and $F_2$~for which the associated Seifert matrices are $M_1$ and $M_2$.
	\end{corollary}
   Such pairs are exactly what the surjectivity direction of Problem~1.6 (Kirby's list K3) asks about: they form an intermediate step toward determining whether a single knot can bound two Seifert surfaces with prescribed $S$-equivalent matrices when $a_{12}+a_{21}\in\{1,3\}$. But whether these two matrices $(a_{12}+a_{21}>3)$ can be realized by the same knot remains an open question.
   
	Our results build on the framework of Aka, Feller, Miller, and Wieser \cite{AkaFellerMillerWieser2023}, who characterized (via Gauss composition of binary quadratic forms) the pairs of Seifert matrices arising from \emph{disjoint} genus-one Seifert surfaces with the same boundary: $s_1$ and $s_2$ arise from such a disjoint pair if and only if $[ax^2+xy+cy^2]^2 * s_1 = s_2$ for some $a, c$ (Theorem~1.1 of \cite{AkaFellerMillerWieser2023}). Our contribution is: we give an explicit geometric band-twist operation and pin down exactly when it produces $\Lambda_1$-equivalent (respectively $S$-equivalent) Seifert matrices, thereby exhibiting concrete genus-one pairs where $S$-equivalence strictly exceeds $\Lambda_1$-equivalence.
			
	The paper is organized as follows. In Section~\ref{construction}, we define the twist operation, prove the $S$-equivalence characterization (Lemma~\ref{lem:a22=0} and Theorem~\ref{thm:sequiv-char}); and establish the Jones polynomial distinction (Lemma~\ref{lem:jones}). In Section~\ref{example}, we give explicit examples. In Section~\ref{high-genus}, we extend the construction to higher genus via connected sums. 
	\section{The Main Construction}\label{construction}

Let $K$ be a genus-one knot whose Seifert matrix is denoted by $M$. The minimal genus Seifert surface of $K$ can be viewed as a disk with two bands attached, drawn as follows (see Figure~\ref{fig:seifert-surface}). Here the dotted lines indicate crossings whose configuration is unspecified.

	\begin{figure}[H]
		\centering

		\tikzset{every picture/.style={line width=0.75pt}} 
		
		\begin{tikzpicture}[x=0.75pt,y=0.75pt,yscale=-0.55,xscale=0.55]
			
			\draw    (114.68,184.64) .. controls (147.48,329.2) and (424.96,275.84) .. (479,152.11) ;
			\draw [shift={(305.74,266.24)}, rotate = 169.85] [color={rgb, 255:red, 0; green, 0; blue, 0 }  ][line width=0.75]    (10.93,-3.29) .. controls (6.95,-1.4) and (3.31,-0.3) .. (0,0) .. controls (3.31,0.3) and (6.95,1.4) .. (10.93,3.29)   ;
			\draw  [dash pattern={on 4.5pt off 4.5pt}]  (278.89,112.34) .. controls (295.38,83.77) and (321.02,43.45) .. (389.54,72.2) ;
			\draw  [dash pattern={on 4.5pt off 4.5pt}]  (294.69,119.07) .. controls (303.86,101.14) and (333.92,58.26) .. (376.81,92.87) ;
			\draw    (191.25,159.27) -- (251.14,154.16) ;
			\draw    (278.89,155.18) -- (305.98,153.77) ;
			\draw    (332.93,154.16) -- (392.82,149.04) ;
			\draw    (420.57,150.07) -- (479,152.11) ;
			\draw  [dash pattern={on 4.5pt off 4.5pt}]  (251.14,154.16) .. controls (257.82,147.69) and (254.31,146.85) .. (266.6,135.91) ;
			\draw  [dash pattern={on 4.5pt off 4.5pt}]  (278.89,155.18) .. controls (281.52,150.22) and (279.77,151.06) .. (283.28,145.17) ;
			\draw    (114.68,184.64) .. controls (105.35,137.51) and (134.19,167.28) .. (164.29,158.88) ;
			\draw  [dash pattern={on 4.5pt off 4.5pt}]  (389.54,72.2) .. controls (411.6,68.89) and (427.71,116.84) .. (420.57,150.07) ;
			\draw  [dash pattern={on 4.5pt off 4.5pt}]  (376.81,92.87) .. controls (398.87,89.56) and (399.96,115.82) .. (392.82,149.04) ;
			\draw  [dash pattern={on 4.5pt off 4.5pt}]  (249.57,106.1) .. controls (270.77,102.79) and (323.37,137.51) .. (332.93,154.16) ;
			\draw  [dash pattern={on 4.5pt off 4.5pt}]  (241.93,119.32) .. controls (267.38,116.02) and (291.13,143.3) .. (305.98,153.77) ;
			\draw  [dash pattern={on 4.5pt off 4.5pt}]  (164.29,158.88) .. controls (166.43,133.38) and (164.73,125.94) .. (198.67,101.14) ;
			\draw  [dash pattern={on 4.5pt off 4.5pt}]  (191.25,159.27) .. controls (185.94,139.99) and (168.97,146.61) .. (202.91,121.8) ;
			\draw  [dash pattern={on 4.5pt off 4.5pt}]  (198.67,101.14) .. controls (232.6,76.33) and (237.69,101.96) .. (249.57,106.1) ;
			\draw  [dash pattern={on 4.5pt off 4.5pt}]  (202.91,121.8) .. controls (209.69,112.3) and (240.23,112.71) .. (241.93,119.32) ;

		\end{tikzpicture}
		
		\caption{Seifert surface of a genus-one knot}
		\label{fig:seifert-surface}

	\end{figure}
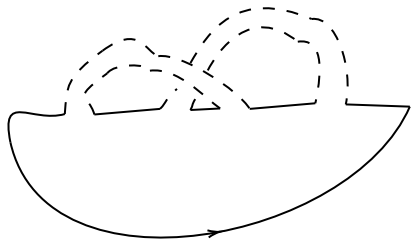

	We now define an operation on a band of the Seifert surface. For convenience, the left band is called the \emph{first band} and the other is called the \emph{second band}. Denote the closed curve through the first band (second band, respectively) by $\alpha_1$ ($\alpha_2$, respectively) when computing the Seifert matrix. Orient $\alpha_1$ and $\alpha_2$ counterclockwise. 
 See Figure \ref{fig:closed-curve}.
	\begin{figure}[H]
		\centering

		\tikzset{every picture/.style={line width=0.75pt}} 
		
		\begin{tikzpicture}[x=0.75pt,y=0.75pt,yscale=-0.55,xscale=0.55]
			
			\draw    (94.68,164.64) .. controls (127.48,309.2) and (404.96,255.84) .. (459,132.11) ;
			\draw [shift={(285.74,246.24)}, rotate = 169.85] [color={rgb, 255:red, 0; green, 0; blue, 0 }  ][line width=0.75]    (10.93,-3.29) .. controls (6.95,-1.4) and (3.31,-0.3) .. (0,0) .. controls (3.31,0.3) and (6.95,1.4) .. (10.93,3.29)   ;
			\draw  [dash pattern={on 4.5pt off 4.5pt}]  (258.89,92.34) .. controls (275.38,63.77) and (301.02,23.45) .. (369.54,52.2) ;
			\draw  [dash pattern={on 4.5pt off 4.5pt}]  (274.69,99.07) .. controls (283.86,81.14) and (313.92,38.26) .. (356.81,72.87) ;
			\draw    (171.25,139.27) -- (231.14,134.16) ;
			\draw    (258.89,135.18) -- (285.98,133.77) ;
			\draw    (312.93,134.16) -- (372.82,129.04) ;
			\draw    (400.57,130.07) -- (459,132.11) ;
			\draw  [dash pattern={on 4.5pt off 4.5pt}]  (231.14,134.16) .. controls (237.82,127.69) and (234.31,126.85) .. (246.6,115.91) ;
			\draw  [dash pattern={on 4.5pt off 4.5pt}]  (258.89,135.18) .. controls (261.52,130.22) and (259.77,131.06) .. (263.28,125.17) ;
			\draw    (94.68,164.64) .. controls (85.35,117.51) and (114.19,147.28) .. (144.29,138.88) ;
			\draw  [dash pattern={on 4.5pt off 4.5pt}]  (369.54,52.2) .. controls (391.6,48.89) and (407.71,96.84) .. (400.57,130.07) ;
			\draw  [dash pattern={on 4.5pt off 4.5pt}]  (356.81,72.87) .. controls (378.87,69.56) and (379.96,95.82) .. (372.82,129.04) ;
			\draw  [dash pattern={on 4.5pt off 4.5pt}]  (229.57,86.1) .. controls (250.77,82.79) and (303.37,117.51) .. (312.93,134.16) ;
			\draw  [dash pattern={on 4.5pt off 4.5pt}]  (221.93,99.32) .. controls (247.38,96.02) and (271.13,123.3) .. (285.98,133.77) ;
			\draw  [dash pattern={on 4.5pt off 4.5pt}]  (144.29,138.88) .. controls (146.43,113.38) and (144.73,105.94) .. (178.67,81.14) ;
			\draw  [dash pattern={on 4.5pt off 4.5pt}]  (171.25,139.27) .. controls (165.94,119.99) and (148.97,126.61) .. (182.91,101.8) ;
			\draw  [dash pattern={on 4.5pt off 4.5pt}]  (178.67,81.14) .. controls (212.6,56.33) and (217.69,81.96) .. (229.57,86.1) ;
			\draw  [dash pattern={on 4.5pt off 4.5pt}]  (182.91,101.8) .. controls (189.69,92.3) and (220.23,92.71) .. (221.93,99.32) ;
			\draw    (297,137) .. controls (246,266) and (146,173) .. (154,138) ;
			\draw [shift={(231.18,202.98)}, rotate = 177.32] [color={rgb, 255:red, 0; green, 0; blue, 0 }  ][line width=0.75]    (10.93,-3.29) .. controls (6.95,-1.4) and (3.31,-0.3) .. (0,0) .. controls (3.31,0.3) and (6.95,1.4) .. (10.93,3.29)   ;
			\draw    (388,138) .. controls (337,267) and (237,174) .. (245,139) ;
			\draw [shift={(322.18,203.98)}, rotate = 177.32] [color={rgb, 255:red, 0; green, 0; blue, 0 }  ][line width=0.75]    (10.93,-3.29) .. controls (6.95,-1.4) and (3.31,-0.3) .. (0,0) .. controls (3.31,0.3) and (6.95,1.4) .. (10.93,3.29)   ;
			\draw  [dash pattern={on 0.84pt off 2.51pt}]  (154,138) .. controls (161,77) and (226,60) .. (297,137) ;
			\draw  [dash pattern={on 0.84pt off 2.51pt}]  (245,139) .. controls (249,66) and (400,-16) .. (388,138) ;
			
			\draw (192,199.4) node [anchor=north west][inner sep=0.75pt]    {$\alpha _{1}$};
			\draw (291,205.4) node [anchor=north west][inner sep=0.75pt]    {$\alpha _{2}$};

		\end{tikzpicture}

		\caption{$\alpha_1$ and $\alpha_2$}
		\label{fig:closed-curve}
	\end{figure}
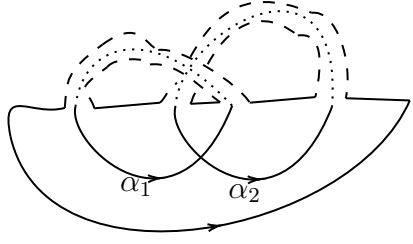

	Cut and move a no-twist part of the first band, and attach $2\ell$ twists on the same band. Denote the resulting knot by $K(\ell,0)$. If the operation is performed on the second band, denote the resulting knot by $K(0,\ell)$. Write $K(0,0) = K$. The genus is unchanged by the operation.
	
	Here $\ell > 0$ means the attached twist is positive, and $\ell < 0$ means the attached twist is negative, once the knot has been given an orientation. See Figure \ref{fig:twist}.

	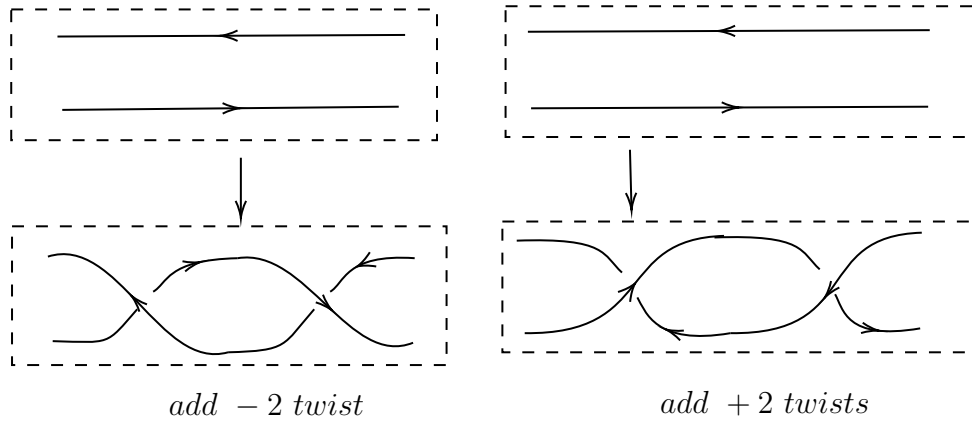
\begin{figure}[H]
		\centering

		\tikzset{every picture/.style={line width=0.75pt}} 
		
		\begin{tikzpicture}[x=0.75pt,y=0.75pt,yscale=-0.6,xscale=0.6]
			
			\draw    (61,177) .. controls (105,163) and (131,249) .. (174,237) ;
			\draw [shift={(113.18,201.35)}, rotate = 46.21] [color={rgb, 255:red, 0; green, 0; blue, 0 }  ][line width=0.75]    (10.93,-3.29) .. controls (6.95,-1.4) and (3.31,-0.3) .. (0,0) .. controls (3.31,0.3) and (6.95,1.4) .. (10.93,3.29)   ;
			\draw    (64,230) .. controls (101,231) and (97,233) .. (117,210) ;
			\draw    (127,199) .. controls (136,194) and (135,179) .. (180,178) ;
			\draw [shift={(156.8,180.35)}, rotate = 165.4] [color={rgb, 255:red, 0; green, 0; blue, 0 }  ][line width=0.75]    (10.93,-3.29) .. controls (6.95,-1.4) and (3.31,-0.3) .. (0,0) .. controls (3.31,0.3) and (6.95,1.4) .. (10.93,3.29)   ;
			\draw    (485,164) .. controls (416,166) and (443,226) .. (360,225) ;
			\draw [shift={(429.21,193.81)}, rotate = 129.18] [color={rgb, 255:red, 0; green, 0; blue, 0 }  ][line width=0.75]    (10.93,-4.9) .. controls (6.95,-2.3) and (3.31,-0.67) .. (0,0) .. controls (3.31,0.67) and (6.95,2.3) .. (10.93,4.9)   ;
			\draw    (355,167) .. controls (412,165) and (409,176) .. (421,187) ;
			\draw    (430.65,202.59) .. controls (435.7,211.57) and (442,233) .. (489,225) ;
			\draw [shift={(448.4,222.61)}, rotate = 20.57] [color={rgb, 255:red, 0; green, 0; blue, 0 }  ][line width=0.75]    (10.93,-3.29) .. controls (6.95,-1.4) and (3.31,-0.3) .. (0,0) .. controls (3.31,0.3) and (6.95,1.4) .. (10.93,3.29)   ;
			\draw    (67,39) -- (285,38) ;
			\draw [shift={(169,38.53)}, rotate = 359.74] [color={rgb, 255:red, 0; green, 0; blue, 0 }  ][line width=0.75]    (10.93,-3.29) .. controls (6.95,-1.4) and (3.31,-0.3) .. (0,0) .. controls (3.31,0.3) and (6.95,1.4) .. (10.93,3.29)   ;
			\draw    (70,85) -- (281,83) ;
			\draw [shift={(181.5,83.94)}, rotate = 179.46] [color={rgb, 255:red, 0; green, 0; blue, 0 }  ][line width=0.75]    (10.93,-3.29) .. controls (6.95,-1.4) and (3.31,-0.3) .. (0,0) .. controls (3.31,0.3) and (6.95,1.4) .. (10.93,3.29)   ;
			\draw    (362,36) -- (614,35) ;
			\draw [shift={(481,35.53)}, rotate = 359.77] [color={rgb, 255:red, 0; green, 0; blue, 0 }  ][line width=0.75]    (10.93,-3.29) .. controls (6.95,-1.4) and (3.31,-0.3) .. (0,0) .. controls (3.31,0.3) and (6.95,1.4) .. (10.93,3.29)   ;
			\draw    (364,84) -- (613,83) ;
			\draw [shift={(494.5,83.48)}, rotate = 179.77] [color={rgb, 255:red, 0; green, 0; blue, 0 }  ][line width=0.75]    (10.93,-3.29) .. controls (6.95,-1.4) and (3.31,-0.3) .. (0,0) .. controls (3.31,0.3) and (6.95,1.4) .. (10.93,3.29)   ;
			\draw  [dash pattern={on 4.5pt off 4.5pt}] (38,22) -- (305,22) -- (305,104) -- (38,104) -- cycle ;
			\draw  [dash pattern={on 4.5pt off 4.5pt}] (38.5,158) -- (311,158) -- (311,245) -- (38.5,245) -- cycle ;
			\draw  [dash pattern={on 4.5pt off 4.5pt}] (348,20) -- (646,20) -- (646,102) -- (348,102) -- cycle ;
			\draw  [dash pattern={on 4.5pt off 4.5pt}] (346,155) -- (645,155) -- (645,237) -- (346,237) -- cycle ;
			\draw    (182,115) -- (182,151) ;
			\draw [shift={(182,153)}, rotate = 270] [color={rgb, 255:red, 0; green, 0; blue, 0 }  ][line width=0.75]    (10.93,-3.29) .. controls (6.95,-1.4) and (3.31,-0.3) .. (0,0) .. controls (3.31,0.3) and (6.95,1.4) .. (10.93,3.29)   ;
			\draw    (426,110) -- (426.95,146) ;
			\draw [shift={(427,148)}, rotate = 268.49] [color={rgb, 255:red, 0; green, 0; blue, 0 }  ][line width=0.75]    (10.93,-3.29) .. controls (6.95,-1.4) and (3.31,-0.3) .. (0,0) .. controls (3.31,0.3) and (6.95,1.4) .. (10.93,3.29)   ;
			\draw    (180,178) .. controls (223,171) and (250,247) .. (290,231) ;
			\draw [shift={(239.2,210.76)}, rotate = 223.18] [color={rgb, 255:red, 0; green, 0; blue, 0 }  ][line width=0.75]    (10.93,-3.29) .. controls (6.95,-1.4) and (3.31,-0.3) .. (0,0) .. controls (3.31,0.3) and (6.95,1.4) .. (10.93,3.29)   ;
			\draw    (174,237) .. controls (212,236) and (208,233) .. (228,210) ;
			\draw    (238,199) .. controls (247,194) and (247,175) .. (291,178) ;
			\draw [shift={(254.92,183.2)}, rotate = 335.79] [color={rgb, 255:red, 0; green, 0; blue, 0 }  ][line width=0.75]    (10.93,-4.9) .. controls (6.95,-2.3) and (3.31,-0.67) .. (0,0) .. controls (3.31,0.67) and (6.95,2.3) .. (10.93,4.9)   ;
			\draw    (609,162) .. controls (540,164) and (572,226) .. (489,225) ;
			\draw [shift={(547.98,202.94)}, rotate = 308.04] [color={rgb, 255:red, 0; green, 0; blue, 0 }  ][line width=0.75]    (10.93,-3.29) .. controls (6.95,-1.4) and (3.31,-0.3) .. (0,0) .. controls (3.31,0.3) and (6.95,1.4) .. (10.93,3.29)   ;
			\draw    (479,165) .. controls (536,163) and (533,174) .. (545,185) ;
			\draw    (554.65,200.59) .. controls (559.7,209.57) and (561,230) .. (608,222) ;
			\draw [shift={(582.22,223.49)}, rotate = 189.43] [color={rgb, 255:red, 0; green, 0; blue, 0 }  ][line width=0.75]    (10.93,-3.29) .. controls (6.95,-1.4) and (3.31,-0.3) .. (0,0) .. controls (3.31,0.3) and (6.95,1.4) .. (10.93,3.29)   ;
			
			\draw (135,260.4) node [anchor=north west][inner sep=0.75pt]    {$add\ -2\ twists$};
			\draw (443,259.4) node [anchor=north west][inner sep=0.75pt]    {$add\ +2\ twists$};

		\end{tikzpicture}

		\caption{add positive twists and negative twists}
		\label{fig:twist}
	\end{figure}

For example, the operation from $K$ to $K(-1,0)$ is illustrated in Figure \ref{fig:example-operation}.
	
	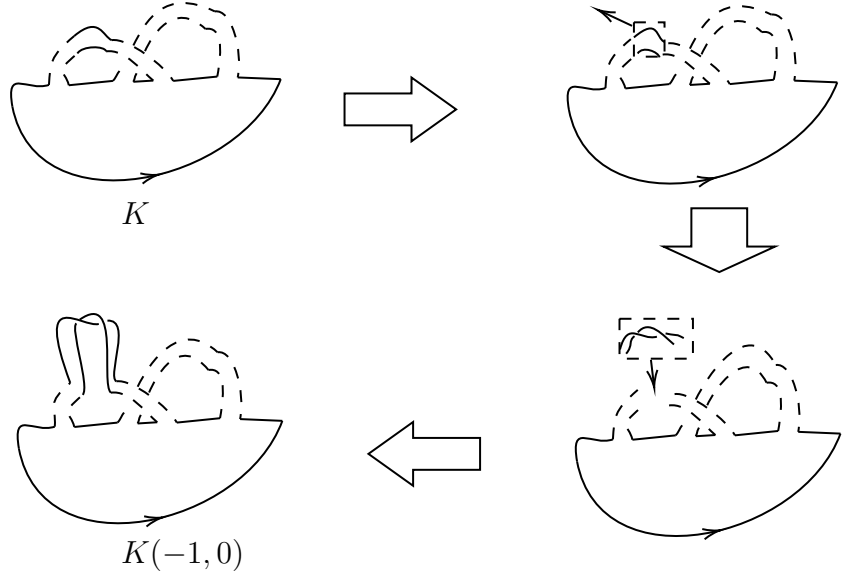
\begin{figure}[H]

		\centering
		\tikzset{every picture/.style={line width=0.75pt}} 
		
		\begin{tikzpicture}[x=0.75pt,y=0.75pt,yscale=-0.8,xscale=0.8]
			
			\draw    (422.72,89.33) .. controls (437.87,166.17) and (566.04,137.81) .. (591,72.04) ;
			\draw [shift={(514.31,131.92)}, rotate = 167.34] [color={rgb, 255:red, 0; green, 0; blue, 0 }  ][line width=0.75]    (10.93,-3.29) .. controls (6.95,-1.4) and (3.31,-0.3) .. (0,0) .. controls (3.31,0.3) and (6.95,1.4) .. (10.93,3.29)   ;
			\draw  [dash pattern={on 4.5pt off 4.5pt}]  (498.57,50.9) .. controls (506.18,35.72) and (518.03,14.29) .. (549.68,29.57) ;
			\draw  [dash pattern={on 4.5pt off 4.5pt}]  (505.86,54.48) .. controls (510.1,44.95) and (523.99,22.16) .. (543.8,40.55) ;
			\draw    (458.09,75.85) -- (485.75,73.13) ;
			\draw    (498.57,73.67) -- (511.08,72.92) ;
			\draw    (523.53,73.13) -- (551.19,70.41) ;
			\draw    (564.01,70.96) -- (591,72.04) ;
			\draw  [dash pattern={on 4.5pt off 4.5pt}]  (485.75,73.13) .. controls (488.84,69.69) and (487.22,69.25) .. (492.89,63.43) ;
			\draw  [dash pattern={on 4.5pt off 4.5pt}]  (498.57,73.67) .. controls (499.78,71.04) and (498.97,71.48) .. (500.59,68.35) ;
			\draw    (422.72,89.33) .. controls (418.41,64.28) and (431.73,80.1) .. (445.64,75.64) ;
			\draw  [dash pattern={on 4.5pt off 4.5pt}]  (549.68,29.57) .. controls (559.87,27.81) and (567.31,53.3) .. (564.01,70.96) ;
			\draw  [dash pattern={on 4.5pt off 4.5pt}]  (543.8,40.55) .. controls (553.99,38.8) and (554.49,52.75) .. (551.19,70.41) ;
			\draw  [dash pattern={on 4.5pt off 4.5pt}]  (485.02,47.59) .. controls (494.82,45.83) and (519.11,64.28) .. (523.53,73.13) ;
			\draw  [dash pattern={on 4.5pt off 4.5pt}]  (481.5,54.62) .. controls (493.25,52.86) and (504.22,67.36) .. (511.08,72.92) ;
			\draw  [dash pattern={on 4.5pt off 4.5pt}]  (445.64,75.64) .. controls (446.62,62.09) and (445.84,58.13) .. (461.51,44.95) ;
			\draw  [dash pattern={on 4.5pt off 4.5pt}]  (458.09,75.85) .. controls (455.64,65.6) and (447.8,69.12) .. (463.47,55.93) ;
			\draw    (463.47,43.85) .. controls (479.15,30.67) and (477.19,44.51) .. (482.67,46.71) ;
			\draw    (465.82,53.96) .. controls (468.96,48.9) and (478.75,52.86) .. (479.54,56.37) ;
			\draw    (462.3,36.6) -- (442.56,27.34) ;
			\draw [shift={(440.75,26.49)}, rotate = 25.13] [color={rgb, 255:red, 0; green, 0; blue, 0 }  ][line width=0.75]    (10.93,-3.29) .. controls (6.95,-1.4) and (3.31,-0.3) .. (0,0) .. controls (3.31,0.3) and (6.95,1.4) .. (10.93,3.29)   ;
			\draw  [dash pattern={on 4.5pt off 4.5pt}] (463.47,34.4) -- (482.67,34.4) -- (482.67,55.93) -- (463.47,55.93) -- cycle ;
			\draw    (427.81,309.78) .. controls (442.68,387.54) and (568.49,358.83) .. (593,292.29) ;
			\draw [shift={(517.71,352.88)}, rotate = 166.96] [color={rgb, 255:red, 0; green, 0; blue, 0 }  ][line width=0.75]    (10.93,-3.29) .. controls (6.95,-1.4) and (3.31,-0.3) .. (0,0) .. controls (3.31,0.3) and (6.95,1.4) .. (10.93,3.29)   ;
			\draw    (454.73,239.52) .. controls (459.74,220.85) and (467.36,234.27) .. (477.89,231.53) ;
			\draw    (458.58,242.19) .. controls (463.2,236.85) and (459.74,234.19) .. (463.85,231.18) ;
			\draw  [dash pattern={on 4.5pt off 4.5pt}]  (502.26,270.9) .. controls (509.74,255.53) and (535,217.5) .. (552.44,249.31) ;
			\draw  [dash pattern={on 4.5pt off 4.5pt}]  (509.43,274.52) .. controls (513.59,264.87) and (534,235.5) .. (546.67,260.42) ;
			\draw    (462.53,296.14) -- (489.68,293.39) ;
			\draw    (502.27,293.94) -- (514.55,293.18) ;
			\draw    (526.77,293.39) -- (553.92,290.64) ;
			\draw    (566.51,291.19) -- (593,292.29) ;
			\draw    (465.94,228.84) .. controls (473.11,220.22) and (481.79,230.67) .. (488.58,236.41) ;
			\draw    (483.32,229.47) .. controls (493.2,231.52) and (488.58,229.3) .. (494.74,232.41) ;
			\draw  [dash pattern={on 4.5pt off 4.5pt}]  (489.68,293.39) .. controls (492.71,289.91) and (491.12,289.46) .. (496.69,283.57) ;
			\draw  [dash pattern={on 4.5pt off 4.5pt}]  (502.27,293.94) .. controls (503.46,291.27) and (502.66,291.72) .. (504.25,288.55) ;
			\draw    (427.81,309.78) .. controls (423.58,284.43) and (436.66,300.44) .. (450.3,295.93) ;
			\draw  [dash pattern={on 4.5pt off 4.5pt}]  (552.44,249.31) .. controls (562.44,247.53) and (569.75,273.32) .. (566.51,291.19) ;
			\draw  [dash pattern={on 4.5pt off 4.5pt}]  (546.67,260.42) .. controls (556.67,258.64) and (557.16,272.77) .. (553.92,290.64) ;
			\draw  [dash pattern={on 4.5pt off 4.5pt}]  (488.97,267.54) .. controls (498.59,265.76) and (522.43,284.43) .. (526.77,293.39) ;
			\draw  [dash pattern={on 4.5pt off 4.5pt}]  (485.51,274.65) .. controls (497.05,272.87) and (507.82,287.55) .. (514.55,293.18) ;
			\draw  [dash pattern={on 4.5pt off 4.5pt}]  (450.3,295.93) .. controls (451.27,282.21) and (450.5,278.21) .. (465.89,264.87) ;
			\draw  [dash pattern={on 4.5pt off 4.5pt}]  (462.53,296.14) .. controls (460.12,285.77) and (452.43,289.33) .. (467.81,275.99) ;
			\draw    (474.14,246.04) -- (475.98,261.6) ;
			\draw [shift={(476.22,263.59)}, rotate = 263.24] [color={rgb, 255:red, 0; green, 0; blue, 0 }  ][line width=0.75]    (10.93,-3.29) .. controls (6.95,-1.4) and (3.31,-0.3) .. (0,0) .. controls (3.31,0.3) and (6.95,1.4) .. (10.93,3.29)   ;
			\draw  [dash pattern={on 4.5pt off 4.5pt}] (454.43,220.29) -- (500.48,220.29) -- (500.48,245.31) -- (454.43,245.31) -- cycle ;
			\draw    (77.81,301.78) .. controls (92.68,379.54) and (218.49,350.83) .. (243,284.29) ;
			\draw [shift={(167.71,344.88)}, rotate = 166.96] [color={rgb, 255:red, 0; green, 0; blue, 0 }  ][line width=0.75]    (10.93,-3.29) .. controls (6.95,-1.4) and (3.31,-0.3) .. (0,0) .. controls (3.31,0.3) and (6.95,1.4) .. (10.93,3.29)   ;
			\draw    (110,260.5) .. controls (87,198.5) and (117.36,226.27) .. (127.89,223.53) ;
			\draw    (117.81,267.99) .. controls (122.43,262.65) and (109.74,226.19) .. (113.85,223.18) ;
			\draw  [dash pattern={on 4.5pt off 4.5pt}]  (152.26,262.9) .. controls (159.74,247.53) and (179,219.5) .. (202.44,241.31) ;
			\draw  [dash pattern={on 4.5pt off 4.5pt}]  (159.43,266.52) .. controls (163.59,256.87) and (178,229.5) .. (196.67,252.42) ;
			\draw    (112.53,288.14) -- (139.68,285.39) ;
			\draw    (152.27,285.94) -- (164.55,285.18) ;
			\draw    (176.77,285.39) -- (203.92,282.64) ;
			\draw    (216.51,283.19) -- (243,284.29) ;
			\draw    (115.94,220.84) .. controls (142,202.5) and (128.71,260.91) .. (135.51,266.65) ;
			\draw    (133.32,221.47) .. controls (149,220.5) and (132.81,256.42) .. (138.97,259.54) ;
			\draw  [dash pattern={on 4.5pt off 4.5pt}]  (139.68,285.39) .. controls (142.71,281.91) and (141.12,281.46) .. (146.69,275.57) ;
			\draw  [dash pattern={on 4.5pt off 4.5pt}]  (152.27,285.94) .. controls (153.46,283.27) and (152.66,283.72) .. (154.25,280.55) ;
			\draw    (77.81,301.78) .. controls (73.58,276.43) and (86.66,292.44) .. (100.3,287.93) ;
			\draw  [dash pattern={on 4.5pt off 4.5pt}]  (202.44,241.31) .. controls (212.44,239.53) and (219.75,265.32) .. (216.51,283.19) ;
			\draw  [dash pattern={on 4.5pt off 4.5pt}]  (196.67,252.42) .. controls (206.67,250.64) and (207.16,264.77) .. (203.92,282.64) ;
			\draw  [dash pattern={on 4.5pt off 4.5pt}]  (138.97,259.54) .. controls (148.59,257.76) and (172.43,276.43) .. (176.77,285.39) ;
			\draw  [dash pattern={on 4.5pt off 4.5pt}]  (135.51,266.65) .. controls (147.05,264.87) and (157.82,279.55) .. (164.55,285.18) ;
			\draw  [dash pattern={on 4.5pt off 4.5pt}]  (100.3,287.93) .. controls (101.27,274.21) and (94.61,273.84) .. (110,260.5) ;
			\draw  [dash pattern={on 4.5pt off 4.5pt}]  (112.53,288.14) .. controls (110.12,277.77) and (102.43,281.33) .. (117.81,267.99) ;
			\draw    (73.72,87.33) .. controls (88.87,164.17) and (217.04,135.81) .. (242,70.04) ;
			\draw [shift={(165.31,129.92)}, rotate = 167.34] [color={rgb, 255:red, 0; green, 0; blue, 0 }  ][line width=0.75]    (10.93,-3.29) .. controls (6.95,-1.4) and (3.31,-0.3) .. (0,0) .. controls (3.31,0.3) and (6.95,1.4) .. (10.93,3.29)   ;
			\draw  [dash pattern={on 4.5pt off 4.5pt}]  (149.57,48.9) .. controls (157.18,33.72) and (169.03,12.29) .. (200.68,27.57) ;
			\draw  [dash pattern={on 4.5pt off 4.5pt}]  (156.86,52.48) .. controls (161.1,42.95) and (174.99,20.16) .. (194.8,38.55) ;
			\draw    (109.09,73.85) -- (136.75,71.13) ;
			\draw    (149.57,71.67) -- (162.08,70.92) ;
			\draw    (174.53,71.13) -- (202.19,68.41) ;
			\draw    (215.01,68.96) -- (242,70.04) ;
			\draw  [dash pattern={on 4.5pt off 4.5pt}]  (136.75,71.13) .. controls (139.84,67.69) and (138.22,67.25) .. (143.89,61.43) ;
			\draw  [dash pattern={on 4.5pt off 4.5pt}]  (149.57,71.67) .. controls (150.78,69.04) and (149.97,69.48) .. (151.59,66.35) ;
			\draw    (73.72,87.33) .. controls (69.41,62.28) and (82.73,78.1) .. (96.64,73.64) ;
			\draw  [dash pattern={on 4.5pt off 4.5pt}]  (200.68,27.57) .. controls (210.87,25.81) and (218.31,51.3) .. (215.01,68.96) ;
			\draw  [dash pattern={on 4.5pt off 4.5pt}]  (194.8,38.55) .. controls (204.99,36.8) and (205.49,50.75) .. (202.19,68.41) ;
			\draw  [dash pattern={on 4.5pt off 4.5pt}]  (136.02,45.59) .. controls (145.82,43.83) and (170.11,62.28) .. (174.53,71.13) ;
			\draw  [dash pattern={on 4.5pt off 4.5pt}]  (132.5,52.62) .. controls (144.25,50.86) and (155.22,65.36) .. (162.08,70.92) ;
			\draw  [dash pattern={on 4.5pt off 4.5pt}]  (96.64,73.64) .. controls (97.62,60.09) and (96.84,56.13) .. (112.51,42.95) ;
			\draw  [dash pattern={on 4.5pt off 4.5pt}]  (109.09,73.85) .. controls (106.64,63.6) and (98.8,67.12) .. (114.47,53.93) ;
			\draw    (112.51,42.95) .. controls (128.19,29.77) and (130.54,43.39) .. (136.02,45.59) ;
			\draw    (114.47,53.93) .. controls (117.61,48.88) and (131.71,49.1) .. (132.5,52.62) ;
			\draw   (282,79) -- (324,79) -- (324,69) -- (352,89) -- (324,109) -- (324,99) -- (282,99) -- cycle ;
			\draw   (482,175) -- (499.5,175) -- (499.5,151) -- (534.5,151) -- (534.5,175) -- (552,175) -- (517,191) -- cycle ;
			\draw   (297,305) -- (325,285) -- (325,295) -- (367,295) -- (367,315) -- (325,315) -- (325,325) -- cycle ;
			
			\draw (140,145.4) node [anchor=north west][inner sep=0.75pt]    {$K$};
			\draw (140,356.4) node [anchor=north west][inner sep=0.75pt]    {$K(-1,0)$};

		\end{tikzpicture}

		\caption{The operation from $K$ to $K(-1,0)$}
		\label{fig:example-operation}
		
	\end{figure}

If the operation is performed on different bands, yielding $K(\ell,0)$ and $K(0,\ell)$, whether the two are equivalent depends on the position of the no-twist part that was moved. We focus on the first band throughout.

Given a genus-one knot $K$ and the knot $K(\ell,0)$ obtained after the operation, with their respective minimal genus Seifert surfaces $S$ and $S'$ (each a disk attaching two bands), the Seifert matrix of $K$ is
$
M = \begin{pmatrix} a_{11} & a_{12} \\ a_{21} & a_{22} \end{pmatrix}
$
and that of $K(\ell,0)$ is
$
M' = \begin{pmatrix} a_{11} - \ell & a_{12} \\ a_{21} & a_{22} \end{pmatrix}, 
$
where $\ell\ne 0$.

We will establish the $S$-equivalence of $M $ and $M'$ under suitable conditions.

\begin{lemma}\label{lem:sequiv}
 $M$ and $M'$ are $\Lambda_1$-equivalent (i.e., $TMT^t = M'$ for some $T \in GL_2(\mathbb{Z})$) if and only if $a_{22}=0$ and $(a_{12}+a_{21}) \mid \ell$.
\end{lemma}

\begin{proof}
	We first show that the conditions $a_{22} = 0$ and $(a_{12} + a_{21}) \mid \ell$ are necessary for $M$ and $M'$ to be related by a $\Lambda_1$-operation (which implies $S$-equivalence), and then verify they are sufficient.
	
	\textbf{Necessity.} Suppose there exists
	\[
	T = \begin{pmatrix} b_{11} & b_{12} \\ b_{21} & b_{22} \end{pmatrix} \in GL_2(\mathbb{Z})
	\]
	such that $TMT^t = M'$. Then
	\[
	\det(T)\det(M)\det(T^t) = \det(M'),
	\]
	so $\det(M) = \det(M')$, giving
	\[
	a_{11}a_{22} - a_{12}a_{21} = (a_{11} - \ell)a_{22} - a_{12}a_{21}.
	\]
	Since $\ell \neq 0$, it follows that $a_{22} = 0$.
	
	Second, from $TMT^t = M'$, we obtain
	\[
	\begin{pmatrix} b_{11} & b_{12} \\ b_{21} & b_{22} \end{pmatrix}
	\begin{pmatrix} a_{11} & a_{12} \\ a_{21} & 0 \end{pmatrix}
	\begin{pmatrix} b_{11} & b_{21} \\ b_{12} & b_{22} \end{pmatrix}
	= \begin{pmatrix} a_{11} - \ell & a_{12} \\ a_{21} & 0 \end{pmatrix}.
	\]
	Expanding yields the system of equations:
	\begin{align}
		(b_{11}^2 - 1)a_{11} + b_{11}b_{12}(a_{21} + a_{12}) &= -\ell \tag{1}\\
		b_{11}b_{21}a_{11} + b_{12}b_{21}a_{21} + (b_{11}b_{22} - 1)a_{12} &= 0 \tag{2}\\
		b_{12}b_{21} - b_{11}b_{22} + 1 &= 0 \tag{3}\\
		b_{21}^2 a_{11} + b_{21}b_{22}(a_{21} + a_{12}) &= 0 \tag{4}
	\end{align}
	Here equation (3) follows from $b_{11}b_{21}a_{11}+b_{12}b_{21}a_{21} + b_{11}b_{22}a_{12}  = a_{12} $,  $b_{21}b_{11}a_{11}+b_{11}b_{22}a_{21} + b_{12}b_{21}a_{12} = a_{21}$,  and $|a_{12} - a_{21}| = 1$ ( unless $K$ is unknot by the property of Seifert matrices \cite{Trotter1962}). We can also get $(a_{12}+a_{21})^2=(a_{12} - a_{21})^2-4a_{12}a_{21}=1-4a_{12}a_{21}\ne 0$
	
	Let $s = a_{21} + a_{12}$. From (3), $b_{11}b_{22} = b_{12}b_{21} + 1$, so the system becomes:
	\begin{align}
		(b_{11}^2 - 1)a_{11} + b_{11}b_{12} s &= -\ell \tag{1'}\\
		b_{21}(b_{11}a_{11} + b_{12} s) &= 0 \tag{2'}\\
		b_{12}b_{21} - b_{11}b_{22} + 1 &= 0 \tag{3'}\\
		b_{21}^2 a_{11} + b_{21}b_{22} s &= 0 \tag{4'}
	\end{align}
	
	Since $T \in GL_2(\mathbb{Z})$ exists by assumption, this system must have a solution. There are two cases:
	
	\textit{Case 1: $b_{21} = 0$.} Then $b_{11}b_{22} = 1$, and since $b_{ij} \in \mathbb{Z}$, $b_{11}^2 = 1$. The system reduces to $b_{11}b_{12} s = -\ell$, so $s = a_{12} + a_{21}$ divides $\ell$.
	
	\textit{Case 2: $b_{21} \neq 0$.} Then $b_{21}a_{11} + b_{22} s = 0$ and $b_{11}a_{11} + b_{12} s = 0$. Since $s\ne 0, b_{21}\ne0 $. If $b_{11}\ne0$, then we get
	\[
	\frac{b_{12}}{b_{11}} =  \frac{b_{22}}{b_{21}}= \frac{a_{11}}{-s}.
	\]
	 then  $b_{12}b_{21} - b_{11}b_{22} = 0$, contradicting $T\in \mathrm{GL}_2(\mathbb{Z})$. We can also get the contradiction if $b_{11}=0$.
	
	\textbf{Sufficiency.} Let $a_{22} = 0$ and $s = a_{12} + a_{21}$ divide $\ell$. Setting $b_{21} = 0$, $b_{11} = 1$, $b_{22} = 1$, $b_{12} = -\ell/s$, we obtain
	\[
	T = \begin{pmatrix} 1 & -\ell/s \\ 0 & 1 \end{pmatrix} \in GL_2(\mathbb{Z}),
	\]
	and one verifies that $TMT^t = M'$.
\end{proof}

$M$ and $M'$ are related by a $\Lambda_1$-operation---conjugation by a unimodular matrix. While the $\Lambda_1$-relation is strictly stronger than $S$-equivalence in general, it is natural here because the twist operation on a single band corresponds to a change of homology basis. 

For the $S$-equivalence of $
M = \begin{pmatrix} a_{11} & a_{12} \\ a_{21} & a_{22} \end{pmatrix}
$
and 
$
M' = \begin{pmatrix} a_{11} - \ell & a_{12} \\ a_{21} & a_{22} \end{pmatrix},  \ell \neq 0
$, we obtain the following.

For convenience, we write $M$ and $M'$ are $S$-equivalent as $M\sim_{S}M'$.
\begin{lemma}\label{lem:a22=0}
If $M\sim_{S}M'$, then $a_{22}=0$.
\end{lemma}
\begin{proof}

$S$-equivalent knots have the same Alexander polynomial up to multiplication by $\pm t^{k}$ \cite{Murasugi1965}. For the Seifert matrices as
\[
M=\begin{pmatrix}a_{11}&a_{12}\\ a_{21}&a_{22}\end{pmatrix},
\qquad
M'=\begin{pmatrix}a_{11}-\ell&a_{12}\\ a_{21}&a_{22}\end{pmatrix},
\qquad \ell\neq0,
\]
a direct computation gives
\begin{align*}
\Delta_M(t)&=\det(tM-M^{t})
=\det\begin{pmatrix}a_{11}(t-1)&ta_{12}-a_{21}\\ ta_{21}-a_{12}&a_{22}(t-1)\end{pmatrix}\\
&=(a_{11}a_{22}-a_{12}a_{21})t^{2}+\bigl(a_{12}^{2}+a_{21}^{2}-2a_{11}a_{22}\bigr)t+(a_{11}a_{22}-a_{12}a_{21}),\\
\Delta_{M'}(t)&=\bigl((a_{11}-\ell)a_{22}-a_{12}a_{21}\bigr)t^{2}+\bigl(a_{12}^{2}+a_{21}^{2}-2(a_{11}-\ell)a_{22}\bigr)t\\
&\quad+\bigl((a_{11}-\ell)a_{22}-a_{12}a_{21}\bigr).
\end{align*}
Set $A=a_{11}a_{22}-a_{12}a_{21}$ and $B=a_{12}^{2}+a_{21}^{2}-2a_{11}a_{22}$; then $\Delta_M(t)=At^{2}+Bt+A$ and $\Delta_{M'}(t)=A't^{2}+B't+A'$ with $A'=A-\ell a_{22}$, $B'=B+2\ell a_{22}$. Neither polynomial vanishes identically: if $A=0$ then $a_{11}a_{22}=a_{12}a_{21}$, whence $B=(a_{12}-a_{21})^{2}=1$ by $|a_{12}-a_{21}|=1$ \cite{Trotter1962}, so $\Delta_M(t)=t$; similarly $A'=B'=0$ would force $(a_{12}-a_{21})^{2}=0$, impossible.

Because $S$-equivalent knots share the Alexander polynomial up to $\pm t^{k}$, we have $\Delta_M(t)=\varepsilon t^{k}\Delta_{M'}(t)$ for some $\varepsilon\in\{\pm1\}$ and $k\in\mathbb{Z}$. Both polynomials have degree at most $2$, so $k\in\{-2,-1,0,1,2\}$.
\begin{itemize}
\item $k=2$: then $\Delta_M=\varepsilon t^{2}\Delta_{M'}$, comparing the coefficient of $t$ and the constant term, and we get $A=B=0$ and hence $\Delta_M=0$, impossible.
\item $k=1$: comparing coefficients yields $A=A'=0$ and $B=\varepsilon A'=0$. Since $A-A'=\ell a_{22}$ and $\ell\neq0$, we get $a_{22}=0$; then $B=a_{12}^{2}+a_{21}^{2}=0$ forces $a_{12}=a_{21}=0$, contradicting $|a_{12}-a_{21}|=1$.
\item $k\le -1$: interchanging the roles of $\Delta_M$ and $\Delta_{M'}$ reduces to the case $k\ge1$, already impossible.
\end{itemize}
Thus $k=0$ and $\Delta_M(t)=\varepsilon\Delta_{M'}(t)$.
If $\varepsilon=-1$, comparing the $t^{2}$- and $t$-coefficients gives $A=-A'$ and $B=-B'$. From $A=-A'$ we obtain $2A=\ell a_{22}$; from $B=-B'$ we obtain $B=-\ell a_{22}=-2A$. Substituting $A=a_{11}a_{22}-a_{12}a_{21}$ and $B=a_{12}^{2}+a_{21}^{2}-2a_{11}a_{22}$ into $B=-2A$ yields $(a_{12}-a_{21})^{2}=0$, contradicting $|a_{12}-a_{21}|=1$.
Hence $\varepsilon=+1$ and $\Delta_M=\Delta_{M'}$. Comparing the $t^{2}$-coefficients gives $A=A'$, i.e.\ $\ell a_{22}=0$, and since $\ell\neq0$ we conclude $a_{22}=0$.

\end{proof}

Recall that for the square discriminant $D=s^{2}$ ($s$ odd), \cite[Proposition~6.8]{AkaFellerMillerWieser2023} provides the group isomorphism
\[
\phi_{s}:(\mathbb{Z}/s\mathbb{Z})^{\times}\xrightarrow{\sim}\mathcal{G}^{+}_{s^{2}},\qquad \phi_{s}(a)=\bigl[ax^{2}+sxy\bigr],
\]
under which Gauss composition corresponds to multiplication of coefficients modulo $s$. Set
\[
H_{s}:=\phi_{s}^{-1}\bigl(\mathcal{S}^{+}_{s^{2}}\bigr)\le(\mathbb{Z}/s\mathbb{Z})^{\times},
\]

\begin{theorem}\label{thm:sequiv-char}
Let $gcd(a_{11},a_{12}+a_{21})=1$. Then 
$M\sim_{S}M'$ if and only if
\[\ell\equiv a_{11}(1-h)\mod (a_{12}+a_{21}).\tag{$\dagger$}\]
for some $h\in H_s$.
\end{theorem}
\begin{proof}
\noindent\textit{Step 1. $''\Rightarrow''$}

Put $s=a_{12}+a_{21}$; the associated IBQFs then have square discriminant $s^{2}$. For a genus-one Seifert matrix $M=\begin{pmatrix}a_{11}&a_{12}\\ a_{21}&0\end{pmatrix}$ in a symplectic basis $\{e_{1},e_{2}\}$, the intersection form $s_M$ gives the IBQF
\[
q_M(x,y)=s_M(xe_{1}+ye_{2},xe_{1}+ye_{2})=a_{11}x^{2}+(a_{12}+a_{21})xy,
\]
the $y^{2}$-term vanishing because $a_{22}=0$. Thus
\[
q(x,y)=a_{11}x^{2}+sxy,\qquad q'(x,y)=(a_{11}-\ell)x^{2}+sxy,
\]
both of discriminant $D=s^{2}-4a_{11}\cdot0=s^{2}$.

By \cite[Theorem 14.1]{Bandes74} and \cite[Proposition 2.9]{Miller22}, two IBQFs in $\mathcal{Q}_{D}^{+}$ are $S$-equivalent if and only if they differ by the action of an element of $\mathcal{S}^{+}_{D}\subseteq\mathcal{G}^{+}_{D}$ via Gauss composition. Since $\mathcal{G}^{+}_{D}$ is abelian under this law (denoted~$*$), this means
\[
M\sim_{S}M'\quad\Longleftrightarrow\quad [q']=\alpha*[q]\ \text{for some }\alpha\in\mathcal{S}^{+}_{s^{2}}. \tag{*}
\]

We now \emph{add} the hypothesis $\gcd(a_{11},s)=1$, so that $q$ is primitive and lies in the domain of the isomorphism below. And $h$ is a element of $H_{s}$ precisely when $\phi_{s}(h)\in\mathcal{S}^{+}_{s^{2}}$. Since $[q]=\phi_{s}(a_{11})$ and $[q']=\phi_{s}(a_{11}-\ell)$, applying $\phi_{s}$ to~$(*)$ gives $\phi_{s}(a_{11}-\ell)=\alpha*\phi_{s}(a_{11})$ for some $\alpha\in\mathcal{S}^{+}_{s^{2}}$. Equivalently, with $h:=\phi_{s}^{-1}(\alpha)\in H_{s}$,
\[
a_{11}-\ell\equiv ha_{11}\pmod s,
\]
i.e.
\[
\ell\equiv a_{11}(1-h)\pmod s.
\]
If $M\sim_{S}M'$, then by~$(*)$ there exists $\alpha\in\mathcal{S}^{+}_{s^{2}}$ with $[q']=\alpha*[q]$; hence $\phi_{s}(a_{11}-\ell)=\alpha*\phi_{s}(a_{11})$, and putting $h:=\phi_{s}^{-1}(\alpha)\in H_{s}$ yields~$(\dagger)$.

\noindent\textit{Step 2. $''\Leftarrow''$}

Conversely, assume $a_{22}=0,\gcd(a_{11},s)=1,\ell\equiv a_{11}(1-h)\pmod s$ for some $h\in H_{s}$. Then $\alpha:=\phi_{s}(h)\in\mathcal{S}^{+}_{s^{2}}$ by definition of $H_{s}$, and $a_{11}-\ell\equiv ha_{11}\pmod s$ implies $\phi_{s}(a_{11}-\ell)=\phi_{s}(ha_{11})=\phi_{s}(h)*\phi_{s}(a_{11})=\alpha*[q]$. Thus $[q']=\alpha*[q]$ with $\alpha\in\mathcal{S}^{+}_{s^{2}}$, so by \cite[Theorem~5.8]{AkaFellerMillerWieser2023} the matrices $M$ and $M'$ are $S$-equivalent. 
(Primitivity of $q'$ is automatic: $a_{11}-\ell\equiv ha_{11}\pmod s$ with $\gcd(a_{11},s)=\gcd(h,s)=1$ gives $\gcd(a_{11}-\ell,s)=1$.)

\end{proof}

The $\Lambda_{1}$-equivalence criterion of Lemma~\ref{lem:sequiv} requires $a_{22}=0$ and $s\mid\ell$, which is precisely the special case $h=1$ of~$(\dagger)$. By \cite[Proposition~6.8,Lemma~6.9]{AkaFellerMillerWieser2023}, $\mathcal{S}^{+}_{s^{2}}$ is trivial exactly for $s=1$ and $s=3$, and is non-trivial for every odd $s>3$ (it always contains $\phi_{s}(4)\neq\mathrm{id}$). Thus we obtain that 
\begin{corollary}\label{cor:S-lambda1}
If $s=1$ or $3$, then $M$ and $M'$ are $S$-equivalent $\iff$ $M$ and $M'$ are $\Lambda_1$-equivalent.
\end{corollary}
Recall that by \cite[Lemma~6.9]{AkaFellerMillerWieser2023}, when $s\geq 3$, the special classes $[ax^{2}+xy+cy^{2}]$ with $1-4ac=s^{2}$ are exactly the elements $\phi_{s}(2\delta\beta\bmod s)$ where $\delta\mid k+1$, $\beta\mid k$, and $s=2k+1,k\geq 1$. Since $\mathcal{S}^{+}_{s^{2}}$ is generated by the squares of these special classes (by the definition of $\mathcal{S}^{+}_{D}$, see \cite{AkaFellerMillerWieser2023}), the corresponding subgroup $H_{s}=\phi_{s}^{-1}(\mathcal{S}^{+}_{s^{2}})\le(\mathbb{Z}/s\mathbb{Z})^{\times}$ is generated by their squares:
\[
 (2\delta\beta)^{2}\bmod s. \tag{$\ddagger$}
\]
where $\delta\mid k+1,\ \beta\mid k$.

\begin{example}
For the smallest non-trivial case $s=5$ we have $k=2$, so $\delta\in\{1,3\}$ and $\beta\in\{1,2\}$; then $2\delta\beta\in\{2,6,4,12\}\equiv\{2,1,4,2\}\pmod5$, and therefore
\[
H_{5}=\langle 2^{2},1^{2},4^{2}\rangle=\{1,4\}\subset(\mathbb{Z}/5\mathbb{Z})^{\times},
\]
a subgroup of order $2$ (the squares), confirming that $S$-equivalence is strictly broader than $\Lambda_{1}$-equivalence here. 
\end{example}

\begin{remark}
The same statement holds for the pair $\begin{pmatrix} a_{11} & a_{12} \\ a_{21} & a_{22} \end{pmatrix}$ and $\begin{pmatrix} a_{11} & a_{12} \\ a_{21} & a_{22}-\ell \end{pmatrix}$, $\ell\neq0$, with the condition $a_{22}=0$ replaced by $a_{11}=0$.
\end{remark}

\begin{example}
Consider a knot $K$ with Seifert matrix $M = \begin{pmatrix} 0 & 1 \\ 2 & 0 \end{pmatrix}$ and the knot $K(\ell,0)$ obtained from the twist operation. The Seifert matrix of $K(\ell,0)$ is
\[
M' = \begin{pmatrix} -\ell & 1 \\ 2 & 0 \end{pmatrix}.
\]
Since $a_{22}=0$ and $s=a_{12}+a_{21}=1+2=3$ divides $\ell=3k$, Lemma~\ref{lem:sequiv} gives a $\Lambda_1$-operation, so $K$ and $K(\ell,0)$ are $\Lambda_1$- and hence $S$-equivalent. Indeed,
\[
\begin{pmatrix} 1 & -k \\ 0 & 1 \end{pmatrix} M \begin{pmatrix} 1 & 0 \\ -k & 1 \end{pmatrix} = M'.
\]
\end{example}

\begin{example}\label{ex:s5}
The strict inclusion of $\Lambda_1$-equivalence in $S$-equivalence for $s>3$ is witnessed explicitly at $s=5$. Take $a_{11}=1$, $a_{12}=2$, $a_{21}=3$, so that $s=a_{12}+a_{21}=5$ (odd, since $|a_{12}-a_{21}|=1$) and $\gcd(a_{11},s)=\gcd(1,5)=1$. Under the hypothesis $\gcd(a_{11},s)=1$, Theorem~\ref{thm:sequiv-char} reduces the $S$-equivalence condition to the congruence~$(\dagger)$:
\[
\ell\equiv a_{11}(1-h)\pmod s,\qquad h\in H_s.
\]
For $s=5$ one has $k=2$, so $\delta\in\{1,3\}$ and $\beta\in\{1,2\}$; the computation in the proof of Theorem~\ref{thm:sequiv-char} then gives $H_5=\{1,4\}$, the subgroup of squares in $(\mathbb Z/5\mathbb Z)^\times$. The two branches are:
\begin{itemize}
\item $h=1$: $\ell\equiv 1(1-1)\equiv 0\pmod 5$, i.e.\ $5\mid\ell$; this is exactly the $\Lambda_1$-equivalence case of Lemma~\ref{lem:sequiv}.
\item $h=4$: $\ell\equiv 1(1-4)\equiv -3\equiv 2\pmod 5$.
\end{itemize}
Choosing $\ell=2$ (so $5\nmid\ell$), the matrices
\[
M=\begin{pmatrix}1&2\\3&0\end{pmatrix},
\qquad
M'=\begin{pmatrix}-1&2\\3&0\end{pmatrix}
\]
satisfy~$(\dagger)$ with $h=4$, hence are $S$-equivalent by Theorem~\ref{thm:sequiv-char}; but $s=5\nmid 2=\ell$, so they are \emph{not} $\Lambda_1$-equivalent by Lemma~\ref{lem:sequiv}. Thus this is a concrete pair of genus-one Seifert matrices that are $S$-equivalent but not $\Lambda_1$-equivalent, confirming that for $s=5$ the two equivalence relations do not coincide.
\end{example}

\begin{lemma}\label{lem:jones}
	If the Jones polynomial $V(K)$ of $K$ is not equal to $1$, then $K(\ell,0)$ is not equivalent to $K$.
\end{lemma}

\begin{proof}
Let $\ell > 0$. By the skein relation for the Jones polynomial (see Figure~\ref{fig:skein})

	\begin{figure}[H]
		\centering
		\tikzset{every picture/.style={line width=0.75pt}} 
		
		\begin{tikzpicture}[x=0.75pt,y=0.75pt,yscale=-0.75,xscale=0.75]
			\draw    (39.89,863.54) .. controls (37.26,938.89) and (196.14,937.48) .. (194.83,858.3) ;
			\draw    (39.89,863.54) .. controls (41.27,846.54) and (60.9,864.34) .. (61.68,849.43) ;
			\draw    (67.72,850.44) .. controls (66.67,861.32) and (95.03,863.14) .. (96.09,850.24) ;
			\draw    (101.86,852.26) .. controls (96.09,863.14) and (132.85,862.93) .. (131.27,851.85) ;
			\draw    (136.79,850.64) .. controls (133.9,863.74) and (166.73,861.52) .. (166.73,851.05) ;
			\draw    (173.82,851.05) .. controls (173.03,864.34) and (193.51,850.04) .. (194.83,858.3) ;
			\draw  [dash pattern={on 4.5pt off 4.5pt}]  (68.82,849.03) .. controls (75.73,841.56) and (79.47,842.11) .. (88.1,836.86) ;
			\draw  [dash pattern={on 4.5pt off 4.5pt}]  (61.68,849.43) .. controls (65.19,840.66) and (64.29,835.62) .. (80.69,832.71) ;
			\draw  [dash pattern={on 4.5pt off 4.5pt}]  (96.09,850.24) .. controls (108.95,826.06) and (153.6,806.32) .. (173.82,851.05) ;
			\draw  [dash pattern={on 4.5pt off 4.5pt}]  (101.86,852.26) .. controls (114.21,826.06) and (157.27,823.64) .. (166.73,851.05) ;
			\draw   (56.59,860.28) -- (49.66,856.05) -- (56.69,852.3) ;
			\draw  [dash pattern={on 4.5pt off 4.5pt}]  (103.13,832.44) .. controls (117.8,837.97) and (129.55,846.25) .. (131.27,851.85) ;
			\draw  [dash pattern={on 4.5pt off 4.5pt}]  (108.02,826.63) .. controls (120.54,825.8) and (134.24,845.02) .. (136.79,850.64) ;
			\draw    (88.1,836.86) .. controls (87.88,836.59) and (97.38,821.93) .. (99.97,831.88) ;
			\draw    (80.69,832.71) .. controls (81.8,830.3) and (90.47,829.67) .. (90.47,831.06) ;
			\draw    (93.06,833.54) .. controls (97.95,841.29) and (104.28,822.76) .. (106.29,826.91) ;
			\draw     ;
			\draw   (322.17,898) -- (326.47,891.29) -- (329.43,898.69) ;
			\draw    (261.89,783.87) .. controls (259.26,859.22) and (418.14,857.81) .. (416.83,778.63) ;
			\draw    (261.89,783.87) .. controls (263.27,766.87) and (282.9,784.68) .. (283.68,769.77) ;
			\draw    (289.72,770.78) .. controls (288.67,781.66) and (317.03,783.47) .. (318.09,770.57) ;
			\draw    (323.86,772.59) .. controls (318.09,783.47) and (354.85,783.27) .. (353.27,772.19) ;
			\draw    (358.79,770.98) .. controls (355.9,784.07) and (388.73,781.86) .. (388.73,771.38) ;
			\draw    (395.82,771.38) .. controls (395.03,784.68) and (415.51,770.37) .. (416.83,778.63) ;
			\draw  [dash pattern={on 4.5pt off 4.5pt}]  (290.82,769.36) .. controls (297.73,761.89) and (303.57,755.81) .. (308.17,756.67) ;
			\draw  [dash pattern={on 4.5pt off 4.5pt}]  (283.68,769.77) .. controls (287.19,760.99) and (286.29,755.95) .. (302.69,753.05) ;
			\draw  [dash pattern={on 4.5pt off 4.5pt}]  (318.09,770.57) .. controls (330.95,746.4) and (375.6,726.65) .. (395.82,771.38) ;
			\draw  [dash pattern={on 4.5pt off 4.5pt}]  (323.86,772.59) .. controls (336.21,746.4) and (379.27,743.98) .. (388.73,771.38) ;
			\draw   (276.7,778.33) -- (268.8,776.48) -- (274.3,770.72) ;
			\draw  [dash pattern={on 4.5pt off 4.5pt}]  (325.13,752.77) .. controls (339.8,758.3) and (351.55,766.59) .. (353.27,772.19) ;
			\draw  [dash pattern={on 4.5pt off 4.5pt}]  (330.02,746.96) .. controls (342.54,746.14) and (356.24,765.35) .. (358.79,770.98) ;
			\draw    (307.5,751.67) .. controls (311.17,746.67) and (317.17,744) .. (321.5,750) ;
			\draw    (302.69,753.05) .. controls (327.5,760.67) and (326.28,743.09) .. (328.29,747.24) ;
			\draw     ;
			\draw    (274.22,929.24) .. controls (271.59,1004.59) and (430.47,1003.18) .. (429.16,924) ;
			\draw    (274.22,929.24) .. controls (275.61,912.24) and (295.23,930.05) .. (296.02,915.14) ;
			\draw    (302.06,916.14) .. controls (301.01,927.02) and (329.37,928.84) .. (330.42,915.94) ;
			\draw    (336.2,917.96) .. controls (330.42,928.84) and (367.18,928.64) .. (365.61,917.55) ;
			\draw    (371.12,916.35) .. controls (368.23,929.44) and (401.06,927.23) .. (401.06,916.75) ;
			\draw    (408.15,916.75) .. controls (407.36,930.05) and (427.85,915.74) .. (429.16,924) ;
			\draw  [dash pattern={on 4.5pt off 4.5pt}]  (303.16,914.73) .. controls (305.17,910) and (305.5,910.33) .. (308.17,908.67) ;
			\draw  [dash pattern={on 4.5pt off 4.5pt}]  (296.02,915.14) .. controls (299.52,906.36) and (306.5,902.67) .. (307.5,902.67) ;
			\draw  [dash pattern={on 4.5pt off 4.5pt}]  (330.42,915.94) .. controls (343.29,891.76) and (387.93,872.02) .. (408.15,916.75) ;
			\draw  [dash pattern={on 4.5pt off 4.5pt}]  (336.2,917.96) .. controls (348.54,891.76) and (391.61,889.35) .. (401.06,916.75) ;
			\draw   (289.05,925.93) -- (283.7,921.98) -- (288.83,917.96) ;
			\draw  [dash pattern={on 4.5pt off 4.5pt}]  (337.46,898.14) .. controls (352.14,903.67) and (363.88,911.95) .. (365.61,917.55) ;
			\draw  [dash pattern={on 4.5pt off 4.5pt}]  (342.35,892.33) .. controls (354.87,891.5) and (368.57,910.72) .. (371.12,916.35) ;
			\draw    (327.39,899.25) .. controls (321.83,893.67) and (332.04,884.3) .. (334.63,894.25) ;
			\draw    (307.5,902.33) .. controls (308.61,899.92) and (316.17,902.33) .. (308.17,908.67) ;
			\draw    (327.39,899.25) .. controls (332.28,906.99) and (338.61,888.46) .. (340.63,892.61) ;
			\draw     ;
			\draw    (205.17,854) -- (244.92,823.22) ;
			\draw [shift={(246.5,822)}, rotate = 142.25] [color={rgb, 255:red, 0; green, 0; blue, 0 }  ][line width=0.75]    (10.93,-3.29) .. controls (6.95,-1.4) and (3.31,-0.3) .. (0,0) .. controls (3.31,0.3) and (6.95,1.4) .. (10.93,3.29)   ;
			\draw    (206.17,893.33) -- (241.93,921.43) ;
			\draw [shift={(243.5,922.67)}, rotate = 218.16] [color={rgb, 255:red, 0; green, 0; blue, 0 }  ][line width=0.75]    (10.93,-3.29) .. controls (6.95,-1.4) and (3.31,-0.3) .. (0,0) .. controls (3.31,0.3) and (6.95,1.4) .. (10.93,3.29)   ;
			\draw    (440.17,928.67) -- (489.33,928.67) ;
			\draw [shift={(491.33,928.67)}, rotate = 180] [color={rgb, 255:red, 0; green, 0; blue, 0 }  ][line width=0.75]    (10.93,-3.29) .. controls (6.95,-1.4) and (3.31,-0.3) .. (0,0) .. controls (3.31,0.3) and (6.95,1.4) .. (10.93,3.29)   ;
			\draw   (530.83,929) .. controls (530.83,915.19) and (542.03,904) .. (555.83,904) .. controls (569.64,904) and (580.83,915.19) .. (580.83,929) .. controls (580.83,942.81) and (569.64,954) .. (555.83,954) .. controls (542.03,954) and (530.83,942.81) .. (530.83,929) -- cycle ;
			\draw   (508.67,929) .. controls (508.67,902.95) and (529.78,881.83) .. (555.83,881.83) .. controls (581.88,881.83) and (603,902.95) .. (603,929) .. controls (603,955.05) and (581.88,976.17) .. (555.83,976.17) .. controls (529.78,976.17) and (508.67,955.05) .. (508.67,929) -- cycle ;
			\draw   (514.81,916.79) -- (508.37,923.95) -- (506.71,914.46) ;
			\draw   (530.04,916.39) -- (537.58,910.8) -- (536.26,920.1) ;
			
			\draw (90,922.67) node [anchor=north west][inner sep=0.75pt]    {$D_{+}=K_{(\ell,0)}$};
			\draw (320,999) node [anchor=north west][inner sep=0.75pt]    {$D_{0}$};
			\draw (325.67,843) node [anchor=north west][inner sep=0.75pt]    {$D_{-}=K_{(\ell-1,0)}$};

		\end{tikzpicture}

		\caption{Skein relation}
			\label{fig:skein}
	\end{figure}
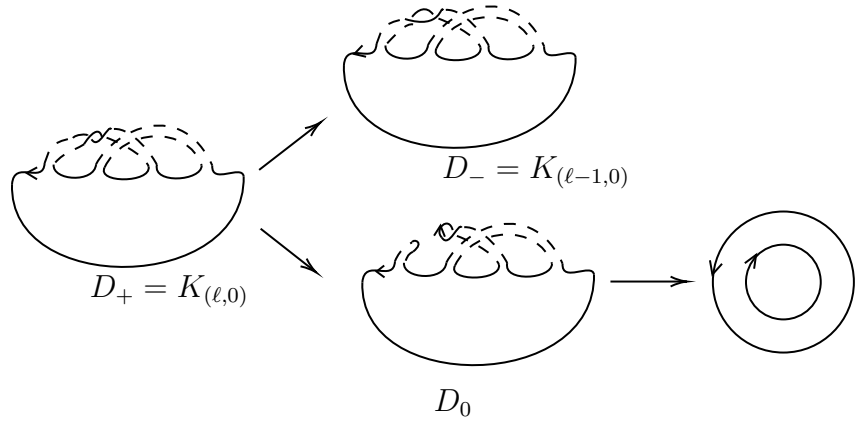

	we have
	\[
	t^{-1}V(K(\ell,0)) - tV(K(\ell-1,0)) = (t^{1/2} - t^{-1/2})(-(t^{-1/2} + t^{1/2})) = t^{-1} - t,
	\]
	so
	\[
	V(K(\ell,0)) = t^2 V(K(\ell-1,0)) - t^2 + 1.
	\]
	By induction,
	\begin{align*}
		V(K(\ell,0)) &= t^{2\ell} V(K) + (1 - t^2)(t^{2(\ell-1)} + \cdots + t^2 + 1) \\
		&= t^{2\ell} V(K) + (1 - t^2) \cdot \frac{1 - (t^2)^{\ell}}{1 - t^2} \\
		&= t^{2\ell} V(K) + 1 - t^{2\ell}.
	\end{align*}
	If $V(K(\ell,0)) = V(K)$, then $t^{2\ell}V(K) + 1 - t^{2\ell} = V(K)$, which gives $(t^{2\ell} - 1)(V(K) - 1) = 0$. Since $t^{2\ell} - 1 \neq 0$ as a polynomial, we must have $V(K) = 1$, contradicting our hypothesis.
	
	For $\ell < 0$, a similar argument yields $V(K(\ell,0)) = V(K)$ implies $V(K) = 1$, again a contradiction.
\end{proof}
By Lemma~\ref{lem:sequiv} and Lemma~\ref{lem:jones}, we obtain the following.

\begin{theorem}\label{thm:main-body}
	For a knot $K$ with Jones polynomial $V(K)\neq 1$ and Seifert matrix satisfying $a_{22}=0$ together with the congruence of Theorem~\ref{thm:sequiv-char} (in particular whenever $s=a_{12}+a_{21}\mid\ell$), $K(\ell,0)$ is $S$-equivalent to $K$ but not equivalent to $K$. Thus the twist operation produces $S$-equivalent but inequivalent genus-one knots---an infinite family parametrized by the admissible values of $\ell$.
\end{theorem}

\begin{remark}
	Lemma~\ref{lem:sequiv}, Lemma~\ref{lem:jones}, and Theorem~\ref{thm:main-body} are stated for $K(\ell,0)$. The same results hold for $K(0,\ell)$, with the condition $a_{22} = 0$ replaced by $a_{11} = 0$.
\end{remark}

	\section{Example}\label{example}

In this section, we provide explicit examples of our construction. The Seifert surface $S$ of a genus-one knot can be viewed as a disk attaching two bands. Denote the closed curves through the bands by $\alpha_1, \alpha_2$ (generators of $H_1(S)$) when computing the Seifert matrix. We introduce a convenient notation $\lambda(n,m,p)$ for the examples below; the relationship to the $K(\ell,0)$ notation of Section~\ref{construction} is explained after Definition~\ref{def:lambda}.

\begin{definition}
	Define the types of double crossing between two bands as shown in Figure~\ref{fig:double-crossing}, where the dashed line represents a band and the solid line with arrows is the closed curve $\alpha_1$ or $\alpha_2$.
\end{definition}

		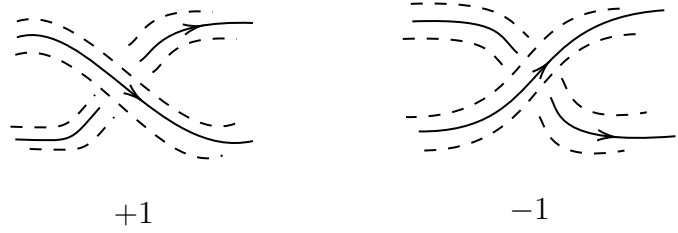
\begin{figure}[H]
			\centering

			\tikzset{every picture/.style={line width=0.75pt}} 
			
			\begin{tikzpicture}[x=0.75pt,y=0.75pt,yscale=-0.8,xscale=0.8]
				
				\draw  [dash pattern={on 4.5pt off 4.5pt}]  (103,55) .. controls (147,41) and (200,135) .. (240,119) ;
				\draw  [dash pattern={on 4.5pt off 4.5pt}]  (111,135) .. controls (148,136) and (144,138) .. (164,115) ;
				\draw  [dash pattern={on 4.5pt off 4.5pt}]  (188,87) .. controls (197,82) and (197,63) .. (241,66) ;
				\draw    (509,48) .. controls (418,54) and (438,125) .. (355,124) ;
				\draw [shift={(436.17,80.9)}, rotate = 133.17] [color={rgb, 255:red, 0; green, 0; blue, 0 }  ][line width=0.75]    (10.93,-3.29) .. controls (6.95,-1.4) and (3.31,-0.3) .. (0,0) .. controls (3.31,0.3) and (6.95,1.4) .. (10.93,3.29)   ;
				\draw    (351,55) .. controls (408,53) and (405,64) .. (417,75) ;
				\draw    (437.65,102.59) .. controls (442.7,111.57) and (441,133) .. (516,127) ;
				\draw [shift={(478.12,127.52)}, rotate = 186.13] [color={rgb, 255:red, 0; green, 0; blue, 0 }  ][line width=0.75]    (10.93,-3.29) .. controls (6.95,-1.4) and (3.31,-0.3) .. (0,0) .. controls (3.31,0.3) and (6.95,1.4) .. (10.93,3.29)   ;
				\draw  [dash pattern={on 4.5pt off 4.5pt}]  (102,76) .. controls (146,62) and (192,155) .. (232,139) ;
				\draw    (103,65) .. controls (147,51) and (193,143) .. (251,130) ;
				\draw [shift={(180.56,104.15)}, rotate = 218.32] [color={rgb, 255:red, 0; green, 0; blue, 0 }  ][line width=0.75]    (10.93,-3.29) .. controls (6.95,-1.4) and (3.31,-0.3) .. (0,0) .. controls (3.31,0.3) and (6.95,1.4) .. (10.93,3.29)   ;
				\draw    (102,129) .. controls (139,130) and (135,132) .. (155,109) ;
				\draw  [dash pattern={on 4.5pt off 4.5pt}]  (99,121) .. controls (136,122) and (132,124) .. (152,101) ;
				\draw    (180,80) .. controls (189,75) and (192,54) .. (251,56) ;
				\draw [shift={(218.42,57.8)}, rotate = 167.89] [color={rgb, 255:red, 0; green, 0; blue, 0 }  ][line width=0.75]    (10.93,-3.29) .. controls (6.95,-1.4) and (3.31,-0.3) .. (0,0) .. controls (3.31,0.3) and (6.95,1.4) .. (10.93,3.29)   ;
				\draw  [dash pattern={on 4.5pt off 4.5pt}]  (173,70) .. controls (182,65) and (182,46) .. (226,49) ;
				\draw  [dash pattern={on 4.5pt off 4.5pt}]  (345,66) .. controls (402,64) and (399,75) .. (411,86) ;
				\draw  [dash pattern={on 4.5pt off 4.5pt}]  (350,44) .. controls (407,42) and (412,54) .. (424,65) ;
				\draw  [dash pattern={on 4.5pt off 4.5pt}]  (430.65,114.59) .. controls (435.7,123.57) and (437,144) .. (484,136) ;
				\draw  [dash pattern={on 4.5pt off 4.5pt}]  (444.65,90.59) .. controls (449.7,99.57) and (451,120) .. (498,112) ;
				\draw  [dash pattern={on 4.5pt off 4.5pt}]  (491,63) .. controls (422,65) and (437,137) .. (354,136) ;
				\draw  [dash pattern={on 4.5pt off 4.5pt}]  (481,47) .. controls (412,49) and (430,116) .. (347,115) ;
				
				\draw (162,166.4) node [anchor=north west][inner sep=0.75pt]    {$+1$};
				\draw (410,164.4) node [anchor=north west][inner sep=0.75pt]    {$-1$};

			\end{tikzpicture}
			
			\caption{positive double crossing and negative double crossing}
			\label{fig:double-crossing}
		\end{figure} 

For convenience, as in Section~\ref{construction}, the left band is called the first band and the other is called the second band. We denote the closed curve through the first band (second band, respectively) by $\alpha_1$ ($\alpha_2$, respectively), oriented counterclockwise.

\begin{definition}\label{def:lambda}
	Define a type of knot $\lambda(n,m,p)$ as follows:
	\begin{itemize}
		\item $n$ is the number of twists of the first band of $S$, with $|n|$ an even number (since the Seifert surface is orientable);
		\item $m$ is the number of twists of the second band, with $|m|$ also an even number;
		\item $p$ is the number of double crossings between the two bands, with $|p|$ odd and $|p| \geq 3$ (so that the knot is non-trivial);
		\item There are no other types of crossings between the two bands.
	\end{itemize}
\end{definition}

	\begin{example}
		The following Figure~\ref{fig:example-lambda} shows some examples of $\lambda(n,m,p)$.
	
		\begin{figure}[H]
			\centering
			\tikzset{every picture/.style={line width=0.4pt}} 
			
			\begin{tikzpicture}[x=0.75pt,y=0.75pt,yscale=-0.8,xscale=0.8]
				
				\draw    (54.27,41.96) .. controls (79.18,26.89) and (154.55,26.64) .. (138.36,55.28) ;
				\draw    (61.12,47.99) .. controls (75.14,35.93) and (139.29,33.67) .. (132.75,54.02) ;
				\draw    (128.39,34.42) .. controls (153.31,19.35) and (228.99,36.18) .. (195.66,106.03) ;
				\draw    (137.42,38.44) .. controls (165.14,32.67) and (205.01,48.24) .. (188.81,91.21) ;
				\draw    (147.39,75.88) .. controls (166.08,50.5) and (110.02,57.54) .. (125.28,43.97) ;
				\draw    (137.42,72.61) .. controls (153.93,54.78) and (98.49,62.56) .. (118.43,41.96) ;
				\draw    (153.62,89.7) .. controls (167.01,69.1) and (109.39,76.89) .. (134.31,61.81) ;
				\draw    (143.34,90.45) .. controls (163.27,74.88) and (98.8,79.65) .. (129.95,58.8) ;
				\draw    (90.71,89.95) .. controls (113.75,89.7) and (121.85,90.2) .. (132.13,77.89) ;
				\draw    (143.34,90.45) .. controls (136.49,92.97) and (127.77,87.44) .. (141.78,79.65) ;
				\draw    (90.71,89.95) .. controls (41.81,82.16) and (36.2,63.07) .. (61.12,47.99) ;
				\draw    (67.35,130.4) .. controls (26.24,116.08) and (10.98,63.82) .. (54.27,41.96) ;
				\draw    (67.35,130.4) .. controls (96.94,138.94) and (169.19,141.71) .. (195.66,106.03) ;
				\draw    (153.62,89.7) .. controls (149.88,95.48) and (181.96,101.76) .. (188.81,91.21) ;
				\draw    (336.23,38.16) .. controls (357.2,39.24) and (365.22,48.82) .. (360.29,61.28) ;
				\draw    (324.66,51.98) .. controls (339.42,38.7) and (355.92,46.2) .. (353.46,58.48) ;
				\draw    (349.08,38.16) .. controls (370.96,21.9) and (437.41,40.05) .. (408.15,115.38) ;
				\draw    (357.01,42.49) .. controls (381.35,36.26) and (416.35,53.06) .. (402.13,99.4) ;
				\draw    (366.31,82.96) .. controls (383.26,57.13) and (333.95,62.18) .. (347.35,47.55) ;
				\draw    (357.56,78.62) .. controls (372.05,59.38) and (319.28,65.07) .. (341.52,44.66) ;
				\draw    (371.23,97.77) .. controls (382.99,75.55) and (340.61,84.13) .. (356.65,67.96) ;
				\draw    (362.21,98.58) .. controls (379.71,81.78) and (323.11,86.93) .. (350.45,64.44) ;
				\draw    (316,98.04) .. controls (336.23,97.77) and (343.34,98.31) .. (352.36,85.04) ;
				\draw    (362.21,98.58) .. controls (356.19,101.29) and (348.54,95.33) .. (360.84,86.93) ;
				\draw    (316,98.04) .. controls (253.01,99.67) and (236.33,75.55) .. (258.21,59.29) ;
				\draw    (295.49,141.67) .. controls (237.7,133.27) and (211.86,78.13) .. (249.87,54.55) ;
				\draw    (295.49,141.67) .. controls (358.02,142.75) and (391.38,148.17) .. (408.15,115.38) ;
				\draw    (371.23,97.77) .. controls (367.95,104) and (396.12,110.78) .. (402.13,99.4) ;
				\draw    (273.79,52.25) .. controls (283.91,36.26) and (279.54,70.13) .. (292.66,54.42) ;
				\draw    (294.3,51.44) .. controls (304.6,36.22) and (305.24,65.53) .. (315.08,51.98) ;
				\draw    (285.01,50.89) .. controls (294.03,37.89) and (289.52,68.51) .. (303.87,53.06) ;
				\draw    (307.15,49.27) .. controls (315.49,38.7) and (312.94,58.16) .. (324.66,51.98) ;
				\draw    (317.27,48.18) .. controls (319.05,43.44) and (331.08,37.34) .. (336.23,38.16) ;
				\draw    (249.87,54.55) .. controls (265.32,43.04) and (256.84,71.49) .. (270.51,56.04) ;
				\draw    (261.76,53.87) .. controls (272.43,38.43) and (272.97,70.13) .. (282.27,56.04) ;
				\draw    (473.74,51.81) .. controls (494.25,37.91) and (556.3,37.68) .. (542.97,64.08) ;
				\draw    (479.38,57.36) .. controls (490.92,46.25) and (543.73,44.16) .. (538.35,62.92) ;
				\draw    (601.93,36.15) .. controls (643.57,41.71) and (645.01,104.71) .. (597.32,131.95) ;
				\draw    (596.4,39.67) .. controls (617.11,46.52) and (632.7,98.59) .. (587.06,98.59) ;
				\draw    (550.4,83.07) .. controls (565.78,59.68) and (519.63,66.17) .. (532.2,53.66) ;
				\draw    (542.2,80.06) .. controls (555.78,63.62) and (510.15,70.8) .. (526.56,51.81) ;
				\draw    (555.53,95.81) .. controls (566.55,76.82) and (519.12,84) .. (539.63,70.1) ;
				\draw    (547.07,96.51) .. controls (563.48,82.15) and (510.4,86.55) .. (536.04,67.32) ;
				\draw    (503.74,96.05) .. controls (522.71,95.81) and (529.38,96.28) .. (537.84,84.93) ;
				\draw    (547.07,96.51) .. controls (541.43,98.83) and (534.25,93.73) .. (545.79,86.55) ;
				\draw    (503.74,96.05) .. controls (463.49,88.87) and (458.87,71.26) .. (479.38,57.36) ;
				\draw    (484.51,133.34) .. controls (450.67,120.13) and (438.1,71.96) .. (473.74,51.81) ;
				\draw    (484.51,133.34) .. controls (508.87,141.21) and (575.27,145.38) .. (597.32,131.95) ;
				\draw    (555.53,95.81) .. controls (552.45,101.14) and (575.27,99.98) .. (587.06,98.59) ;
				\draw    (531.99,44.12) .. controls (534.86,40.04) and (539.58,36.7) .. (544.5,41.34) ;
				\draw    (585.82,40.62) .. controls (597.63,47.64) and (587.37,34.48) .. (601.93,36.15) ;
				\draw    (576.3,40.97) .. controls (587.33,48.55) and (580.19,33.55) .. (592.5,37.07) ;
				\draw    (555.47,41.4) .. controls (565.7,48.37) and (561.12,32.44) .. (571.78,38.74) ;
				\draw    (565.52,41.27) .. controls (580.72,47.52) and (568.3,31.33) .. (582.65,38.37) ;
				\draw    (546.5,42.99) .. controls (558.48,49.07) and (547.09,34.41) .. (561.99,38.16) ;
				\draw    (539.38,46.9) .. controls (550.66,44.86) and (539.79,35.41) .. (551.89,38) ;
				\draw   (57.78,119.62) -- (62.97,128.92) -- (52.35,128.03) ;
				\draw   (264.98,125.62) -- (270.17,134.92) -- (259.55,134.03) ;
				\draw   (488.06,128.4) -- (495,136.48) -- (484.42,137.73) ;
				\draw (59.33,156.33) node [anchor=north west][inner sep=0.75pt]    {$\lambda(0,0,3)$};
				\draw (282,155.67) node [anchor=north west][inner sep=0.75pt]    {$\lambda(-6,0,3)$};
				\draw (513.67,158.67) node [anchor=north west][inner sep=0.75pt]    {$\lambda(0,6,3)$};
			\end{tikzpicture}
			\caption{Some example of $\lambda(n,m,p)$}
			\label{fig:example-lambda}
		\end{figure}
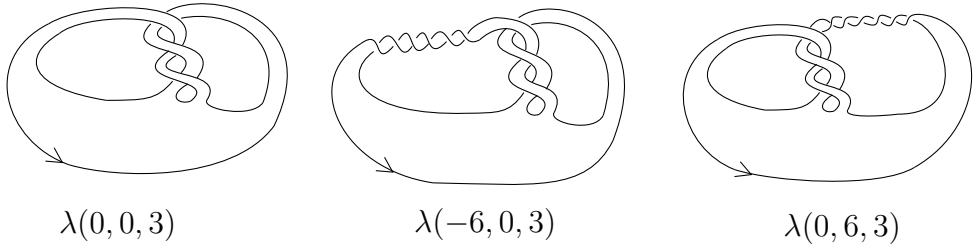
		
	\end{example}

$\lambda(2\ell, 0, 3)$ is exactly the knot obtained by applying the operation of Section~\ref{construction} to $K = \lambda(0,0,3)$ on the first band with parameter $\ell$; i.e., in the notation of Section~\ref{construction}, $\lambda(2\ell, 0, 3) = K(\ell,0)$ where $K = \lambda(0,0,3)$. Similarly, $\lambda(0, 2\ell, 3) = K(0,\ell)$ for the same $K$.

\subsection{$S$-equivalence}

We compute the Seifert matrices $M(\lambda(n,m,p))$ of the knot $\lambda(n,m,p)$:
\[
M(\lambda(0,0,3)) = \begin{pmatrix} 0 & 2 \\ 1 & 0 \end{pmatrix},
\]
\[
M(\lambda(6,0,3)) = \begin{pmatrix} -3 & 2 \\ 1 & 0 \end{pmatrix}, \quad
M(\lambda(-6,0,3)) = \begin{pmatrix} 3 & 2 \\ 1 & 0 \end{pmatrix},
\]
\[
M(\lambda(0,6,3)) = \begin{pmatrix} 0 & 2 \\ 1 & -3 \end{pmatrix}, \quad
M(\lambda(0,-6,3)) = \begin{pmatrix} 0 & 2 \\ 1 & 3 \end{pmatrix}.
\]

We verify that:
\[
\begin{pmatrix} 1 & -1 \\ 0 & 1 \end{pmatrix} M(\lambda(0,0,3)) \begin{pmatrix} 1 & 0 \\ -1 & 1 \end{pmatrix} = M(\lambda(6,0,3)),
\]
\[
\begin{pmatrix} 1 & 1 \\ 0 & 1 \end{pmatrix} M(\lambda(0,0,3)) \begin{pmatrix} 1 & 0 \\ 1 & 1 \end{pmatrix} = M(\lambda(-6,0,3)),
\]
\[
\begin{pmatrix} 1 & 0 \\ -1 & 1 \end{pmatrix} M(\lambda(0,0,3)) \begin{pmatrix} 1 & -1 \\ 0 & 1 \end{pmatrix} = M(\lambda(0,6,3)),
\]
\[
\begin{pmatrix} 1 & 0 \\ 1 & 1 \end{pmatrix} M(\lambda(0,0,3)) \begin{pmatrix} 1 & 1 \\ 0 & 1 \end{pmatrix} = M(\lambda(0,-6,3)).
\]

Thus $K = \lambda(0,0,3)$, $K(3,0)$, $K(0,3)$, $K(-3,0)$, and $K(0,-3)$ are all $S$-equivalent to each other.

\subsection{Jones Polynomial Computation}

The knot $\lambda(0,0,3)$ is $9_{46}$ in Rolfsen's knot table \cite{Rolfsen1976}. Its Jones polynomial, computed via the standard skein relation, is
\[
V(\lambda(0,0,3)) = -\frac{1}{t} + \frac{1}{t^2} - \frac{2}{t^3} + \frac{1}{t^4} - \frac{1}{t^5} + \frac{1}{t^6} + 2.
\]

By Lemma~\ref{lem:jones}, we obtain:
\[
V(\lambda(-6,0,3)) = \frac{1}{t^6} - \frac{1}{t^7} + \frac{1}{t^8} - \frac{2}{t^9} + \frac{1}{t^{10}} - \frac{1}{t^{11}} + \frac{1}{t^{12}} + 1,
\]
\[
V(\lambda(0,-6,3)) = \frac{1}{t^6} - \frac{1}{t^7} + \frac{1}{t^8} - \frac{2}{t^9} + \frac{1}{t^{10}} - \frac{1}{t^{11}} + \frac{1}{t^{12}} + 1,
\]
\[
V(\lambda(6,0,3)) = t^6 - t^5 + t^4 - 2t^3 + t^2 - t + 2,
\]
\[
V(\lambda(0,6,3)) = t^6 - t^5 + t^4 - 2t^3 + t^2 - t + 2.
\]

We obtain
\[
V(\lambda(0,0,3)) \neq V(\lambda(6,0,3)) = V(\lambda(0,6,3)),
\]
\[
V(\lambda(0,0,3)) \neq V(\lambda(-6,0,3)) = V(\lambda(0,-6,3)).
\]

Thus $\lambda(0,0,3)$ is not equivalent to $K(3,0)$, $K(0,3)$, $K(-3,0)$, or $K(0,-3)$. Thus the twist operation on $9_{46}$ produces four distinct $S$-equivalent companion knots, all distinguished from $9_{46}$ by the Jones polynomial.

	\section{A remark on higher genus and further directions}\label{high-genus}
	
	Our construction extends to higher genus knots. We illustrate this with examples and then discuss open questions.
	
	\begin{example}
		Consider $K_1 = K \# K$ and $K_2 = K(-3,0) \# K$, where $K = \lambda(0,0,3)$ in the notation of Section~\ref{construction}. $K(-3,0)$ is exactly $\lambda(-6,0,3)$.  See Figure~ \ref{fig:connect-sum}. 
		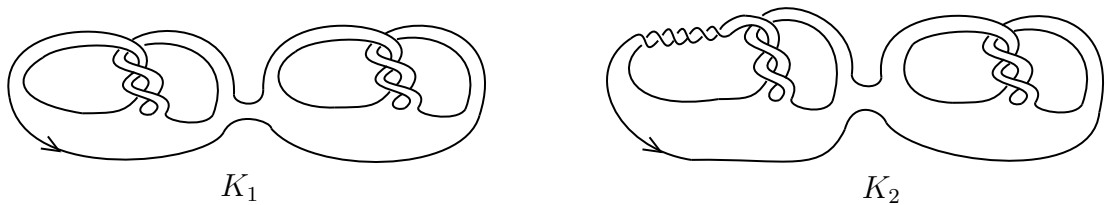
\begin{figure}[H]
			\centering
			\tikzset{every picture/.style={line width=0.75pt}} 
			
			\begin{tikzpicture}[x=0.75pt,y=0.75pt,yscale=-0.7,xscale=0.7]
				
				\draw    (38.67,3280.67) .. controls (30.67,3241.33) and (108.67,3244.67) .. (100.75,3266.78) ;
				\draw    (53.83,3262.37) .. controls (62.34,3255.08) and (101.32,3253.71) .. (97.34,3266.02) ;
				\draw    (94.7,3254.17) .. controls (109.83,3245.05) and (142.45,3250.9) .. (140.99,3279.21) ;
				\draw    (100.18,3256.6) .. controls (117.02,3253.1) and (141.24,3262.52) .. (131.4,3288.51) ;
				\draw    (106.24,3279.24) .. controls (117.59,3263.89) and (83.53,3268.15) .. (92.8,3259.94) ;
				\draw    (100.18,3277.27) .. controls (110.21,3266.48) and (76.53,3271.19) .. (88.64,3258.72) ;
				\draw    (110.02,3287.6) .. controls (118.16,3275.14) and (83.16,3279.85) .. (98.29,3270.73) ;
				\draw    (103.78,3288.05) .. controls (115.89,3278.63) and (76.72,3281.52) .. (95.64,3268.91) ;
				\draw    (71.8,3287.75) .. controls (85.8,3287.6) and (90.72,3287.9) .. (96.97,3280.46) ;
				\draw    (103.78,3288.05) .. controls (99.61,3289.57) and (94.32,3286.23) .. (102.83,3281.52) ;
				\draw    (71.8,3287.75) .. controls (42.1,3283.04) and (38.7,3271.49) .. (53.83,3262.37) ;
				\draw    (38.67,3280.67) .. controls (45.67,3320.67) and (129.26,3310.46) .. (136.33,3296.33) ;
				\draw    (110.02,3287.6) .. controls (107.75,3291.09) and (127.24,3294.89) .. (131.4,3288.51) ;
				\draw    (374.18,3243.23) .. controls (386.85,3243.92) and (391.7,3250) .. (388.73,3257.92) ;
				\draw    (367.18,3252.01) .. controls (376.11,3243.58) and (386.08,3248.34) .. (384.59,3256.14) ;
				\draw    (381.95,3243.23) .. controls (395.17,3232.9) and (426.14,3247.04) .. (422.76,3270.17) ;
				\draw    (386.74,3245.99) .. controls (401.45,3242.03) and (422.61,3252.7) .. (414.01,3282.13) ;
				\draw    (392.36,3271.69) .. controls (402.61,3255.28) and (372.8,3258.49) .. (380.9,3249.2) ;
				\draw    (387.07,3268.94) .. controls (395.83,3256.72) and (363.93,3260.33) .. (377.38,3247.36) ;
				\draw    (395.34,3281.1) .. controls (402.44,3266.99) and (376.83,3272.44) .. (386.52,3262.17) ;
				\draw    (389.88,3281.62) .. controls (400.46,3270.95) and (366.25,3274.22) .. (382.78,3259.93) ;
				\draw    (361.95,3281.27) .. controls (374.18,3281.1) and (378.48,3281.45) .. (383.93,3273.01) ;
				\draw    (389.88,3281.62) .. controls (386.25,3283.34) and (381.62,3279.55) .. (389.06,3274.22) ;
				\draw    (361.95,3281.27) .. controls (321.94,3287.7) and (313.8,3266.99) .. (327.02,3256.66) ;
				\draw    (349.55,3308.99) .. controls (314.62,3303.65) and (299,3268.62) .. (321.98,3253.65) ;
				\draw    (349.55,3308.99) .. controls (394.47,3308.28) and (409.48,3315.9) .. (419.62,3295.07) ;
				\draw    (395.34,3281.1) .. controls (393.35,3285.06) and (410.38,3289.36) .. (414.01,3282.13) ;
				\draw    (336.44,3252.18) .. controls (342.56,3242.03) and (339.91,3263.54) .. (347.85,3253.56) ;
				\draw    (348.84,3251.67) .. controls (355.06,3242) and (355.45,3260.62) .. (361.4,3252.01) ;
				\draw    (343.22,3251.32) .. controls (348.67,3243.06) and (345.95,3262.51) .. (354.62,3252.7) ;
				\draw    (356.61,3250.29) .. controls (361.65,3243.58) and (360.1,3255.94) .. (367.18,3252.01) ;
				\draw    (362.72,3249.6) .. controls (363.8,3246.59) and (371.07,3242.72) .. (374.18,3243.23) ;
				\draw    (321.98,3253.65) .. controls (331.32,3246.33) and (326.19,3264.4) .. (334.46,3254.59) ;
				\draw    (329.17,3253.22) .. controls (335.61,3243.4) and (335.95,3263.54) .. (341.57,3254.59) ;
				\draw    (154.35,3276.79) .. controls (153.38,3234.69) and (229.44,3245.82) .. (215.92,3264.12) ;
				\draw    (169.01,3259.71) .. controls (177.52,3252.42) and (216.49,3251.05) .. (212.52,3263.36) ;
				\draw    (209.87,3251.5) .. controls (225.01,3242.39) and (266,3248.33) .. (253,3288.67) ;
				\draw    (215.36,3253.94) .. controls (232.2,3250.44) and (256.41,3259.86) .. (246.57,3285.85) ;
				\draw    (221.41,3276.58) .. controls (232.76,3261.23) and (198.71,3265.49) .. (207.98,3257.28) ;
				\draw    (215.36,3274.6) .. controls (225.38,3263.81) and (191.71,3268.52) .. (203.82,3256.06) ;
				\draw    (225.2,3284.94) .. controls (233.33,3272.48) and (198.33,3277.19) .. (213.47,3268.07) ;
				\draw    (218.95,3285.39) .. controls (231.06,3275.97) and (191.9,3278.86) .. (210.82,3266.25) ;
				\draw    (186.98,3285.09) .. controls (200.98,3284.94) and (205.9,3285.24) .. (212.14,3277.8) ;
				\draw    (218.95,3285.39) .. controls (214.79,3286.91) and (209.49,3283.57) .. (218.01,3278.86) ;
				\draw    (186.98,3285.09) .. controls (161.64,3285.98) and (153.87,3268.83) .. (169.01,3259.71) ;
				\draw    (158,3294.69) .. controls (158.24,3291.55) and (141.68,3285.93) .. (136.33,3296.33) ;
				\draw    (158,3294.69) .. controls (176.67,3316) and (245.6,3315.85) .. (253,3288.67) ;
				\draw    (225.2,3284.94) .. controls (222.92,3288.43) and (242.41,3292.23) .. (246.57,3285.85) ;
				\draw    (436.3,3271.44) .. controls (430.98,3231.8) and (511.01,3239.42) .. (497.32,3261.18) ;
				\draw    (450.64,3256.3) .. controls (459.11,3248.64) and (497.89,3247.2) .. (493.93,3260.13) ;
				\draw    (491.3,3247.68) .. controls (506.36,3238.1) and (546,3242) .. (535.67,3287.33) ;
				\draw    (496.76,3250.23) .. controls (513.51,3246.56) and (537.61,3256.46) .. (527.82,3283.75) ;
				\draw    (502.78,3274.01) .. controls (514.08,3257.89) and (480.19,3262.36) .. (489.42,3253.74) ;
				\draw    (496.76,3271.94) .. controls (506.74,3260.61) and (473.23,3265.56) .. (485.27,3252.47) ;
				\draw    (506.55,3282.79) .. controls (514.64,3269.7) and (479.82,3274.65) .. (494.88,3265.08) ;
				\draw    (500.33,3283.27) .. controls (512.38,3273.38) and (473.42,3276.41) .. (492.24,3263.16) ;
				\draw    (468.52,3282.95) .. controls (482.45,3282.79) and (487.35,3283.11) .. (493.56,3275.29) ;
				\draw    (500.33,3283.27) .. controls (496.19,3284.87) and (490.92,3281.36) .. (499.39,3276.41) ;
				\draw    (468.52,3282.95) .. controls (445,3281.86) and (443.07,3266.61) .. (450.64,3256.3) ;
				\draw    (422.76,3270.17) .. controls (423.24,3276.27) and (435.33,3277.54) .. (436.3,3271.44) ;
				\draw    (438.48,3292.02) .. controls (451,3310) and (532.33,3321.67) .. (535.67,3287.33) ;
				\draw    (506.55,3282.79) .. controls (504.29,3286.46) and (523.68,3290.45) .. (527.82,3283.75) ;
				\draw    (140.99,3279.21) .. controls (140.99,3284.05) and (154.6,3286.22) .. (154.35,3276.79) ;
				\draw    (419.62,3295.07) .. controls (423.48,3280.84) and (436.67,3286.67) .. (438.48,3292.02) ;
				\draw   (56.14,3298.27) -- (61.27,3304.96) -- (52.92,3303.8) ;
				\draw   (330.54,3298.67) -- (335.67,3305.36) -- (327.32,3304.2) ;	
				\draw (133.33,3314.33) node [anchor=north west][inner sep=0.75pt]    {$K_1$};
				\draw (426,3315.33) node [anchor=north west][inner sep=0.75pt]    {$K_2$};
			\end{tikzpicture}
			\caption{The connect sum $\lambda(0,0,3)\sharp \lambda(0,0,3)$ and $\lambda(0,0,3)\sharp \lambda(-6,0,3)$}
			\label{fig:connect-sum}
		\end{figure}
	
	Since $\lambda(0,0,3)$ is $S$-equivalent to $K(-3,0)$, and the Seifert matrix of a connected sum is the block sum of the individual Seifert matrices, the $\Lambda_1$-operation of Lemma~\ref{lem:sequiv} applied to the first block (with the identity on the second) shows that $K_1$ is $S$-equivalent to $K_2$.
	
	For the Jones polynomial of $K_1$ and $K_2$, since the Jones polynomial is multiplicative over connected sums \cite{Jones1987}, we have
	\[
	V(K_1) = V(\lambda(0,0,3)) \times V(\lambda(0,0,3)),
	\]
	and
	\[
	V(K_2) = V(\lambda(-6,0,3)) \times V(\lambda(0,0,3)).
	\]
	By the computation in Section~3.2, $V(K_1) \neq V(K_2)$. Thus $K_1$ is $S$-equivalent to $K_2$ but they are different knots.
	\end{example}
	
	\begin{remark}
	The connected sum construction above is the simplest way to extend our genus-one examples to higher genus, but it is not the only possibility. More generally, let $K$ be a knot of genus $g \geq 1$ with a Seifert surface $S$ of genus $g$. If $S$ contains a genus-one subsurface (obtained, say, by compressing $S$ along a suitable curve) whose Seifert matrix satisfies the conditions of Lemma~\ref{lem:sequiv}, then performing our twist operation on this subsurface yields a knot $K'$ that is $S$-equivalent but not equivalent to $K$. We expect that many higher-genus knots admit such subsurfaces, though we do not attempt a complete characterization here.
	
	In a different direction, Aka, Feller, Miller, and Wieser \cite[Example~7.4]{AkaFellerMillerWieser2023} construct higher-genus examples using direct sums of binary quadratic forms. Their construction yields pairs of genus $g$ Seifert matrices related by Gauss composition, while our construction produces $S$-equivalent but inequivalent knots via explicit twist operations. The relationship between these two approaches---in particular, whether the Gauss composition framework of \cite{AkaFellerMillerWieser2023} can be used to produce $S$-equivalent pairs in higher genus---remains open.
	\end{remark}

	\section*{Acknowledgements}
	The author would like to thank Xiang Liu and Shasha Yang at Hebei Normal University for discussions, and Xiaowei Pang for her encouragement. The project was funded by Science and Technology Project of Hebei Education Department  (Grant No.~QN2023030), Science Foundation of Hebei Normal University (Grant Nos.~L2022B02, L2025ZD04).
	
	\bibliographystyle{alpha}
	\bibliography{reference}

\end{document}